\documentclass[article,onefignum,onetabnum]{siamart171218}

\usepackage{lipsum}
\usepackage{amsfonts}
\usepackage{graphicx}
\usepackage{epstopdf}
\usepackage{algorithmic}
\ifpdf
  \DeclareGraphicsExtensions{.eps,.pdf,.png,.jpg}
\else
  \DeclareGraphicsExtensions{.eps}
\fi

\newsiamremark{remark}{Remark}
\newsiamremark{hypothesis}{Assumptions}
\crefname{hypothesis}{Hypothesis}{Hypotheses}
\newsiamthm{claim}{Claim}

\headers{Relative entropy for a BGK model for 2D Navier-Stokes}{Roberta Bianchini}

\title{Strong convergence of a vector-BGK model\\ to the incompressible Navier-Stokes equations\\
via the relative entropy method %\thanks{Submitted to the editors DATE.
}

\author{Roberta Bianchini\thanks{\'Ecole Normale Sup\'erieure de Lyon, UMPA, ENS-Lyon, 46, all\'ee d'Italie, 69364-Lyon Cedex 07, France and Consiglio Nazionale delle Ricerche, IAC, via dei Taurini 19, I-00185 Rome, Italy. Mail address: roberta.bianchini@ens-lyon.fr}}

\usepackage{amsopn}

\usepackage{xcolor}
\definecolor{orange}{rgb}{1,0.5,0}
\definecolor{amethyst}{rgb}{0.6, 0.4, 0.8}
\definecolor{ngreen}{rgb}{0.0, 0.5, 0.0}
\definecolor{brown}{rgb}{0.43, 0.21, 0.1}
\definecolor{burgundy}{rgb}{0.5, 0.0, 0.13}

\ifpdf
\hypersetup{
  pdftitle={Relative entropy method for a BGK model for the incompressible Navier-Stokes equations in 2D},
  pdfauthor={Roberta Bianchini}
}
\fi

\usepackage{xcolor}

\begin{document}

\maketitle

% REQUIRED
\begin{abstract}
The aim of this paper is to prove the strong convergence of the solutions to a vector-BGK model under the diffusive scaling to the incompressible Navier-Stokes equations on the two-dimensional torus. This result holds in any interval of time $[0, T]$, with $T>0$.
We also provide the global in time uniform boundedness of the solutions to the approximating system. Our argument is based on the use of local in time $H^s$-estimates for the model, established in a previous work, combined with the $L^2$-relative entropy estimate and the interpolation properties of the Sobolev spaces.
\end{abstract}

% REQUIRED

% REQUIRED
%\begin{AMS}
 
%\end{AMS}

\section{Introduction}
\label{Intro}
In this paper we deal with the incompressible Navier-Stokes equations in two space dimensions,
\begin{equation}
\label{real_NS}
\begin{cases}
& \partial_{t}\textbf{u}^{NS}+\nabla \cdot (\textbf{u}^{NS} \otimes \textbf{u}^{NS}) + \nabla P^{NS}={\nu} \Delta \textbf{u}^{NS}, \\
& \nabla \cdot \textbf{u}^{NS}=0, \\
\end{cases}
\end{equation}
with $(t, x) \in [0, +\infty) \times \mathbb{T}^2,$ and initial data

\begin{equation}
\label{real_NS_initial_data}
\textbf{u}^{NS}(0,x)=\textbf{u}_{0}(x), \,\,\,\,\, \quad \,\,\,\,\, \nabla \cdot \textbf{u}_0=0.
\end{equation}

In (\ref{real_NS}), $\textbf{u}^{NS}$ and $\nabla P^{NS}$ are respectively the velocity field and the gradient of the pressure term, and $\nu>0$ is the viscosity coefficient.\\
Here we consider a vector-BGK model for the incompressible Navier-Stokes equations, i.e. a dicrete velocities BGK system endowed with a vectorial structure, whose general formulation has been introduced in \cite{CN}, while further developments were presented in \cite{VBouchut} from the numerical side and in \cite{Bianchini3} from the analytical point of view. Precisely, we study the following five velocities (15 equations) vector-BGK approximation to the incompressible Navier-Stokes equations,
\begin{equation}
\label{BGK_NS}
\begin{cases}
& \partial_{t}f_{1}^{\varepsilon} + \frac{\lambda}{\varepsilon} \partial_{x}f_{1}^{\varepsilon}=\frac{1}{\tau \varepsilon^{2}}(M_{1}(w^{\varepsilon})-f_{1}^{\varepsilon}), \\
& \partial_{t}f_{2}^{\varepsilon} + \frac{\lambda}{\varepsilon} \partial_{y}f_{2}^{\varepsilon}=\frac{1}{\tau \varepsilon^{2}}(M_{2}(w^{\varepsilon})-f_{2}^{\varepsilon}), \\
& \partial_{t}f_{3}^{\varepsilon}-\frac{\lambda}{\varepsilon}\partial_{x}f_{3}^{\varepsilon}=\frac{1}{\tau \varepsilon^{2}}(M_{3}(w^{\varepsilon})-f_{3}^{\varepsilon}), \\
& \partial_{t}f_{4}^{\varepsilon}-\frac{\lambda}{\varepsilon}\partial_{y}f_{4}^{\varepsilon}=\frac{1}{\tau \varepsilon^{2}}(M_{4}(w^{\varepsilon})-f_{4}^{\varepsilon}), \\
& \partial_{t}f_{5}^{\varepsilon}=\frac{1}{\tau \varepsilon^{2}}(M_{5}(w^{\varepsilon})-f_{5}^{\varepsilon}), \\
\end{cases}
\end{equation}
where
\begin{equation}
\label{w_def}
w^\varepsilon=(\rho^\varepsilon, \varepsilon \rho^\varepsilon u_1^\varepsilon, \varepsilon \rho^\varepsilon u_2^\varepsilon)=(\rho^\varepsilon, \varepsilon \rho^\varepsilon \textbf{u}^\varepsilon)=(\rho^\varepsilon, \textbf{q}^\varepsilon)=\sum_{i=1}^5 f_i^\varepsilon.
\end{equation}
Its main properties are as follows:
\begin{itemize}
\item $f_i^\varepsilon, \, M_i(w^\varepsilon),\; i=1, \cdots 5$, are vector-valued functions taking values in $\mathbb{R}^3$;
\item $\rho^\varepsilon(t, x)$ on $\mathbb{R}^+ \times \mathbb{T}^2$ is the approximating density, taking values in $\mathbb{R}^+$;
\item $\textbf{u}^\varepsilon(t, x)=(u_1^\varepsilon(t, x), u_2^\varepsilon(t, x))$ on $\mathbb{R}^+ \times \mathbb{T}^2$ is the approximating vector field, taking values in $\mathbb{R}^2$;
\item the discrete velocities are $\lambda_1=(\lambda, 0), \; \lambda_2=(0, \lambda), \; \lambda_3=(-\lambda, 0),$\\ $\lambda_4=(0, -\lambda), \; \lambda_5=(0,0),$ where $\lambda$ is a positive constant value.
\end{itemize}
Precise compatibility conditions to be satisfied by the constant parameters of the model and the Maxwellian functions, together with their explicit expressions, will be provided in details in Section \ref{Presentation_section}.\\
BGK models were introduced by Bhatnagar, Gross and Krook as a modified
version of the Boltzmann equation, characterized by the relaxation of the collision operator. Since they present most of the basic properties of hydrodynamics, they are considered interesting models even though they do not contain all of the relevant features of the Boltzmann equation. Essentially, vector-BGK models are inspired by the hydrodynamic limits of the Boltzmann equation \cite{Golse1, Golse2, Cercignani, Esposito, Laure}, but later they have been generalized as approximating equations for different kinds of systems. In this regard, one of the main directions has been the approximation of hyperbolic systems with discrete velocities BGK models, as in \cite{Brenier, XinJin, Natalini, Bouchut, Perthame1}. Similar results have been obtained for convection-diffusion systems under the diffusive scaling \cite{Toscani, BGN, Lattanzio, Aregba2, JinPareschi, Laurent}. Originally,
they presented continuous velocities, see \cite{Perthame1}, but later on discrete velocities BGK models inspired by the relaxation method have been introduced, see \cite{Mascia} for a survey. In the spirit of the relaxation approximations, the main advantage of discrete velocities BGK models is to deal with semilinear systems, see \cite{Natalini, BNP, Imene1}.\\
Here we spend few words on our main result and we provide a sketch of the strategy.
We prove the strong convergence in the Sobolev spaces, for any interval of time $[0, T], \; T>0$, of the vector-BGK model presented in (\ref{BGK_NS}) to the incompressible Navier-Stokes equations on the two-dimensioanl torus.
To achieve this result, the novelty relies in using local in time $H^s$-estimates from a previous work, see \cite{Bianchini3}, combined with the $L^2$-relative entropy estimate and the standard interpolation Theorem. More precisely, part of the results of \cite{Bianchini3} provides uniform (in $\varepsilon$) estimates of Gronwall type in the Sobolev spaces, which hold in $[0, T^*]$, where $T^*>0$ is depending on a fixed constant $M>0$ and on the norm of the initial data. These local bounds guarantee the existence, the minimality and dissipative property of the kinetic entropy, i.e. a convex entropy for (\ref{BGK_NS}), see \cite{Bouchut}. Next, the relative entropy allows us to get a precise rate of convergence of the solutions to our model to the Navier-Stokes equations, which holds for $t \in [0, T^*]$, see Theorem \ref{Theorem_L2_estimate}. Thus, the interpolation Theorem for Sobolev spaces applied to the relative entropy estimate provides a bound for the solutions to our system which is much more precise than the previous pessimistic Gronwall type estimates. 
This is the key point in order to close the argument and to prove the strong convergence for all times of the solutions to (\ref{BGK_NS}) to (\ref{real_NS}), together with the global in time boundedness of the approximating solution itself, in Theorem \ref{Theorem_entropy_inequality_NS}.
In particular, Lemma \ref{lemma_derivatives-BGK} plays a crucial role in quantifying the dissipation term coming from the entropy inequality. 
At the best of our understanding, the expansions in Lemma \ref{lemma_derivatives-BGK} are the only way to establish the relative entropy inequality when, as in our case, the explicit dependency of the kinetic entropy on the singular parameter is not known.\\
We point out that we start from initial data in (\ref{initial_conditions_BGK}) that are small perturbation of the Maxwellians and, thanks to the uniform bounds, in the end we prove that everything is done in a bounded set of the densities. This local setting perfectly fits the framework described in \cite{Bouchut}.\\
The relative entropy method, \cite{Dafermos1, DiPerna}, represents 
an efficient mathematical tool for studying stability and limiting process and it is based on a direct calculation of the relative entropy between
a dissipative solution and a conservative smooth solution for the
considered system, which provides a remarkable stability
estimate.  Far from being complete, we collect here a pair of references for hydrodynamic limits \cite{GS, Laure2}. In the context of singular hyperbolic scaled systems, we remind to \cite{Tzavaras}. Let us point out that this procedure has been successfully applied to the vector-BGK model considered in this paper and presented below (\ref{BGK_NS}) to prove its convergence to the isentropic Euler equations under the hyperbolic scaling, see \cite{Sepe}. Again, relative entropy in hyperbolic relaxation has been used for one-dimensional discrete velocities Boltzmann schemes, see \cite{Tzavaras1}, while in the multidimensional case the question in this context seems to be open. On the other hand, the relative entropy method in diffusive relaxation is of course a more delicate issue, being the diffusive limit an order more precise approximation of the starting system in the Chapman-Enskog expansion, see \cite{LSR}. Besides hydrodynamic limits of the Boltzmann equation, our main reference in this framework is \cite{Lattanzio}. In this paper, the authors apply the relative entropy method to the equations of compressible gas dynamics with friction under the diffusive scaling, so obtaining precise estimates coming from the entropy of the limit hyperbolic system. However, in our case, further complications are due to the fact that the explicit dependency of the kinetic entropy on the singular parameter is not known for our model (\ref{BGK_NS}). The BGK framework in \cite{Bouchut} only guarantees the existence of such an entropy, whose expression is defined by means on the inverse function Theorem. This difficulty requires a better understanding of the dissipative terms provided by the entropy inequality in diffusive relaxation, and new ideas are needed with respect to the existing works, for instance \cite{Lattanzio, Tzavaras1}.
\subsection{Plan of the paper}
The paper is organized as follows. In Section \ref{Presentation_section} we introduce the vector-BGK model and provide some preliminary results. Section \ref{REsection} is devoted to the relative entropy inequality and the strong convergence of the model for all times, in the Sobolev spaces. In the last part of this section we also show the global in time boundedness of the solutions to our model.

\section{Presentation of the model, formal limit, and intermediate results}
\label{Presentation_section}
First, we aim at providing a relative entropy inequality for a vector-BGK model approximating the two-dimensional incompressible Navier-Stokes equations. After, this inequality will allow us to extend for long times the local convergence for smooth solutions achieved in \cite{Bianchini3}. Let us introduce the setting that will be taken into account hereafter.\\
Our approximating vector-BGK model has been presented in (\ref{BGK_NS}), together with a list of the main properties. We point out that, in order to get consistency with the incompressible Navier-Stokes equations, the Maxwellian functions $M_i(w^\varepsilon), \; i=1, \cdots, 5,$ need to satisfy the following compatibility conditions:
\begin{itemize}
\item $\sum_{i=1}^5 M_i(w^\varepsilon)=w^\varepsilon$; 
\item $\sum_{i=1}^5 \lambda_{ij} M_i(w^\varepsilon)=A_j(w^\varepsilon), \quad j=1, 2$, with $A_j$ in (\ref{Fluxes_BGK}).
\end{itemize}
We provide here the explicit expressions of the Maxwellian functions
\begin{equation}
\label{Maxwellians}
M_{1,3}(w^{\varepsilon})=aw^{\varepsilon} \pm \frac{A_{1}(w^{\varepsilon})}{2\lambda}, ~~~~~ M_{2,4}(w^\varepsilon)=aw^{\varepsilon} \pm \frac{A_{2}(w^{\varepsilon})}{2\lambda}, ~~~~~ M_{5}(w^{\varepsilon})=(1-4a)w^{\varepsilon},
\end{equation}
\begin{equation}
\label{Fluxes_BGK}
A_{1}(w^{\varepsilon})=\left(\begin{array}{c}
q_{1}^\varepsilon \\
\frac{(q_{1}^\varepsilon)^{2}}{\rho^\varepsilon}+P(\rho^\varepsilon)\\
\frac{q_{1}^\varepsilon q_{2}^\varepsilon}{\rho^\varepsilon}
\end{array}\right), ~~~~~
A_{2}(w^{\varepsilon})=\left(\begin{array}{c}
q_{2}^\varepsilon \\
\frac{q_{1}^\varepsilon q_{2}^\varepsilon}{\rho^\varepsilon} \\
\frac{(q_{2}^\varepsilon)^{2}}{\rho^\varepsilon}+P(\rho^\varepsilon)
\end{array}\right),
\end{equation}
\begin{equation}
\label{Pressure_approximation_BGK}
P(\rho^\varepsilon)=\frac{({(\rho^\varepsilon)^2}-\bar{\rho}^2)}{2\bar{\rho}},
\end{equation}
where $\bar{\rho}>0$ is constant value, and the following constraint has to be satisfied in order to get consistency with respect to (\ref{real_NS}), see \cite{CN}.
\begin{hypothesis}
\label{assumptions_a}
Let us assume
\begin{equation}
\label{a_def_BGK}
a=\frac{\nu}{2\lambda^{2}\tau}, \qquad 0 < a < \frac{1}{4},
\end{equation}
where $\nu$ is the viscosity coefficient in (\ref{real_NS}). Moreover, we also take the parameter $\lambda>0$ big enough, whose lower bound is defined in [\cite{Bianchini3}, Assumption 2]. This is necessary in order to:
\begin{itemize}
\item guarantee the positivity of the symmetrizer in \cite{Bianchini3};
\item satisfy the sub-characteristic condition, i.e. the positivity of the spectrum of the Jacobian matrices of the Maxwellians, see \cite{Bouchut, NataliniCPAM}.
\end{itemize}
\end{hypothesis}
The change of variables introduced in \cite{Bianchini3},
\begin{equation}
\label{variables_BGK}
\begin{aligned}
& w^{\varepsilon}=\sum_{i=1}^{5}f_{i}^{\varepsilon}, \quad
m^{\varepsilon}=\frac{\lambda}{\varepsilon}(f_{1}^{\varepsilon}-f_{3}^{\varepsilon}), \quad \xi^{\varepsilon}=\frac{\lambda}{\varepsilon}(f_{2}^{\varepsilon}-f_{4}^{\varepsilon}), \\
& k^{\varepsilon}=f_{1}^{\varepsilon}+f_{3}^{\varepsilon}, \qquad
h^{\varepsilon}=f_{2}^{\varepsilon}+f_{4}^{\varepsilon}.
\end{aligned}
\end{equation}
allows us to recover the consistency with respect to (\ref{real_NS}) in a simple way at the formal level. This way, the vector-BGK model (\ref{BGK_NS}) reads:
\begin{equation}
\label{BGK_NS_new_variables}
\begin{cases}
& \partial_{t}w^{\varepsilon}+\partial_{x}m^{\varepsilon}+\partial_{y}\xi^{\varepsilon}=0, \\
& \partial_{t}m^{\varepsilon} + \frac{\lambda^{2}}{\varepsilon^{2}}\partial_{x}k^{\varepsilon}=\frac{1}{\tau \varepsilon^{2}}(\frac{A_{1}(w^{\varepsilon})}{\varepsilon}-m^{\varepsilon}), \\
& \partial_{t}\xi^{\varepsilon}+\frac{\lambda^{2}}{\varepsilon^{2}}\partial_{y}h^{\varepsilon}=\frac{1}{\tau \varepsilon^{2}}(\frac{A_{2}(w^{\varepsilon})}{\varepsilon}-\xi^{\varepsilon}), \\
& \partial_{t}k^{\varepsilon}+\partial_{x}m^{\varepsilon}=\frac{1}{\tau \varepsilon^{2}}(2aw^{\varepsilon}-k^{\varepsilon}), \\
& \partial_{t}h^{\varepsilon}+\partial_{y}\xi^{\varepsilon}=\frac{1}{\tau \varepsilon^{2}}(2aw^{\varepsilon}-h^{\varepsilon}). \\
\end{cases}
\end{equation}
Hereafter, we will drop the apex $\varepsilon$ where there is no ambiguity.
Moreover, we denote by $\mathcal{M}_i(w):= {f}_i$ the solutions to system (\ref{BGK_NS}) after taking the limit under the diffusive scaling. 
The relaxation formulation (\ref{BGK_NS_new_variables}) of the system gives 
\begin{equation}
\label{limit_solution_relaxation_variables}
\begin{aligned}
& m=\frac{\lambda}{\varepsilon}({f}_1-{f}_3):=\frac{\lambda}{\varepsilon}(\mathcal{M}_1({w})-\mathcal{M}_3({w}))=\frac{A_1({w})}{\varepsilon}-\tau \lambda^2 \partial_x {k} + O(\varepsilon^2), \\
& {\xi}=\frac{\lambda}{\varepsilon}({f}_2-{f}_4):=\frac{\lambda}{\varepsilon}(\mathcal{M}_2({w})-\mathcal{M}_4({w}))=\frac{A_2({w})}{\varepsilon}-\tau \lambda^2 \partial_y {h} + O(\varepsilon^2), \\
& {k}={f}_1+{f}_2=\mathcal{M}_1(\bar{w})+\mathcal{M}_2({w})=2a {w}+O(\varepsilon^2), \\
& {h}={f}_2+{f}_4=\mathcal{M}_2({w})+\mathcal{M}_4(w)=2a {w}+O(\varepsilon^2).
\end{aligned}
\end{equation}

Recalling that, from Assumptions \ref{assumptions_a} $\nu=2a \tau \lambda^2$, formally we get
$$\partial_t w+\frac{\partial_x A_1(w)}{\varepsilon}+ \frac{\partial_y A_2(w)}{\varepsilon}=\nu \Delta w+O(\varepsilon^2).$$
More explicitly, from the expressions of $w, A_1(w), A_2(w)$ in (\ref{w_def})-(\ref{Fluxes_BGK}),
\begin{align*}
\partial_t \left(\begin{array}{c}
\rho-\bar{\rho}\\
\varepsilon \rho u_1\\
\varepsilon \rho u_2\\
\end{array}\right)&+\partial_x \left(\begin{array}{c}
\rho u_1\\
\varepsilon \rho u_1^2+\frac{\rho^2-\bar{\rho}^2}{2\bar{\rho }\varepsilon}\\\
\varepsilon \rho u_1 u_2
\end{array}\right)+\partial_y \left(\begin{array}{c}
\rho u_2\\
\varepsilon \rho u_1 u_2\\
\varepsilon \rho u_2^2+\frac{\rho^2-\bar{\rho}^2}{2\bar{\rho }\varepsilon}\\
\end{array}\right)\\
&=\nu \Delta \left(\begin{array}{c}
\rho-\bar{\rho}\\
\varepsilon \rho u_1\\
\varepsilon \rho u_2
\end{array}\right)+O(\varepsilon^2).
\end{align*}

Dividing the last two lines by $\varepsilon$, this yields

\begin{equation*}
\begin{cases}
& \partial_t (\rho-\bar{\rho})+\nabla \cdot \textbf{u}=\nu \Delta (\rho-\bar{\rho})+O(\varepsilon), \\
& \partial_t (\rho \textbf{u})+\nabla \cdot (\rho \textbf{u} \otimes \textbf{u})+ \frac{\nabla(\rho^2-\bar{\rho}^2)}{2\bar{\rho} \varepsilon^2}=\nu \Delta (\rho \textbf{u})+O(\varepsilon),
\end{cases}
\end{equation*}

which is the compressible approximation to the incompressible Navier-Stokes equations provided by the scaled isentropic Euler equations.\\
Now we find an expression of the formal limit in terms of the original kinetic variables (\ref{BGK_NS}).
The limit solution is obtained by solving the linear system (\ref{limit_solution_relaxation_variables}) in the unknowns $\mathcal{M}_i(w), \; i=1, \cdots, 5$, so providing
\begin{equation}
\label{Perturbed_Maxwellians}
\begin{aligned}
& {\mathcal{M}}_1(w)=M_1(w)-a \varepsilon \lambda \tau \partial_x w, \\
& {\mathcal{M}}_2(w)=M_2(w)-a \varepsilon \lambda \tau \partial_y w, \\
& {\mathcal{M}}_3(w)=M_3(w)+a \varepsilon \lambda \tau \partial_x w, \\
& {\mathcal{M}}_4(w)=M_4(w)+a \varepsilon \lambda \tau \partial_y w, \\
& {\mathcal{M}}_5(w)=M_5(w).
\end{aligned}
\end{equation}
In order to avoid further complications due to the initial layer, in our convergence proof the two-dimensional vector-BGK model is endowed with the following initial data:
\begin{equation}
\label{initial_conditions_BGK}
f_{i}^\varepsilon(0,x)=\overline{\mathcal{M}}_i(\bar{\rho}, \varepsilon \bar{\rho} \bar{\textbf{u}}_0), \quad i=1, \cdots, 5,
\end{equation}
where $\textbf{u}_0$ is in (\ref{real_NS_initial_data}) and $\bar{\rho}$ is a positive constant value.
According to the theory developed by Bouchut \cite{Bouchut}, the existence of a kinetic entropy for system (\ref{BGK_NS}) is subjected to the existence of a convex entropy for the limit solution to (\ref{BGK_NS}) under the hyperbolic scaling. The hyperbolic parameter of the vector-BGK approximation (\ref{BGK_NS}) is represented by $\tau$ and the limit equations approximated by (\ref{BGK_NS}) in the vanishing parameter of the hyperbolic scaling $\tau$ are the isentropic Euler equations. The convergence of the hyperbolic-scaled system is guaranteed by the structural properties of our vector-BGK model listed before, see \cite{CN}, while a rigorous proof is provided in \cite{Sepe}. A convex entropy for the limit equation in hyperbolic scaling, i.e, the isentropic Euler equations, is given by
\begin{equation}
\label{entropy_isentropic_Euler}
\eta(w^\varepsilon)=\frac{1}{2}\frac{|\textbf{q}^\varepsilon|^2}{ \rho^\varepsilon}+k(\rho^\varepsilon)^2.
\end{equation}
\subsection{Preliminary results}
Here we collect some preliminary results, which hold for local times, essentially due to our previous work \cite{Bianchini3}. Let us start with the following remark.

\begin{remark}
\label{Remark_previous_setting}
We discuss some differences between \cite{Bianchini3} and our current setting.
\begin{itemize}
\item In \cite{Bianchini3}, the compressible pressure $P(\rho^\varepsilon)$ in (\ref{Pressure_approximation_BGK}) is linear. More precisely,
from [\cite{Bianchini3}, (10)],
$$\tilde{P}(\rho^\varepsilon)=\rho^\varepsilon-\bar{\rho}.$$
In this paper, we consider the case of a quadratic pressure $P(\rho^\varepsilon)$ in (\ref{Pressure_approximation_BGK}).
A simple remark shows that, from (\ref{Pressure_approximation_BGK}),
\begin{align*}
P(\rho^\varepsilon)&=\frac{(\rho^\varepsilon)^2-\bar{\rho}^2}{2\bar{\rho}}\\
&=\frac{2\bar{\rho} (\rho^\varepsilon-\bar{\rho})+(\rho^\varepsilon-\bar{\rho})^2}{2\bar{\rho}}\\
&=(\rho^\varepsilon-\bar{\rho})+\frac{(\rho^\varepsilon-\bar{\rho})^2}{2\bar{\rho}}.
\end{align*}
Thus, the estimates in \cite{Bianchini3} still hold here: the quadratic pressure only provides an additional quadratic term in the fifth and the ninth line of the nonlinear vector $N(w+\bar{w})$ in [\cite{Bianchini3}, (26)]. These supplementary quadratic terms can be handled exactly as the other ones in the energy estimates in \cite{Bianchini3}.
However, we point out that the same argument holds exactly in the same way for a general compressible pressure
$$P(\rho^\varepsilon)=
\begin{cases}
& \frac{k}{\gamma-1}[(\rho^\varepsilon)^\gamma-\bar{\rho}^\gamma], \quad \gamma > 1, \\
& k[\rho^\varepsilon log(\rho^\varepsilon)-\bar{\rho}log(\bar{\rho})], \quad \gamma=1,
\end{cases}$$
where $k$ is a positive constant value.

\item In [\cite{Bianchini3}, (18)-(19)], we consider a translated version of the relaxation system (\ref{BGK_NS_new_variables}). Of course this is an equivalent formulation of the approximating model, and since the translation vector $(\bar{\rho}, 0, 0)$ in [\cite{Bianchini3}, (18)] is constant in $t$ and $x$, most of the energy estimates in \cite{Bianchini3} can be used here.

\item A further change of variables, involving the dissipative constant right symmetrizer $\Sigma$ in [\cite{Bianchini3}, (28)] is defined in [\cite{Bianchini3}, (30)]. However, here the energy estimates from \cite{Bianchini3} are expressed in terms of the original relaxation variables (\ref{variables_BGK}) to avoid further complications. The explicit change of variables is written in [\cite{Bianchini3}, (78)].

%\item In \cite{Bianchini3}, the spatial domain of the local in time convergence of system (\ref{BGK_NS}) to the incompressible Navier-Stokes equations (\ref{real_NS}) is the two-dimensional torus $\mathbb{T}^2$.  This was useful for the compactness argument in [Section 5, \cite{Bianchini3}], since the Sobolev embeddings were compact. However, the energy estimates before [\cite{Bianchini3}, Sect. 5] still hold in $\mathbb{T}^2$.
\end{itemize}
\end{remark}

Taking into account Remark \ref{Remark_previous_setting}, we state some results that will be applied below. Hereafter, we denote by $T^\varepsilon$ the maximum time of existence of the solution to the semilinear  vector-BGK approximation (\ref{BGK_NS}) with initial data (\ref{initial_conditions_BGK}), see \cite{Majda}. Of course $T^\varepsilon$ could depend on $\varepsilon$. In the following, we recall and adapt some results from \cite{Bianchini3}, showing that there exist $\varepsilon_0$ and a fixed and positive time $T^*$, independent of $\varepsilon$ and depending on the Sobolev norm of the initial data, such that, for $\varepsilon \le \varepsilon_0$, some local in time $H^s$-estimates on the solutions to the approximating system hold uniformly with respect to $\varepsilon$.
In this context, we consider the constant vector $(\bar{\rho}, 0, 0)$ and the translated variables
\begin{equation}
\label{translated_variables}
\begin{aligned}
& w^*(t, x)=w(t, x)-(\bar{\rho}, 0, 0),\\
& k^*(t, x)=k(t. x)-2a (\bar{\rho}, 0, 0),\\  
& h^*(t, x)=h(t. x)-2a (\bar{\rho}, 0, 0).
\end{aligned}
\end{equation}

\begin{lemma}
\label{lemma_energy_estimates_KRM_paper}
Consider the vector-BGK system (\ref{BGK_NS}) with initial data (\ref{initial_conditions_BGK}), and $\textbf{u}_0$ in (\ref{real_NS_initial_data}) beloging to $H^s(\mathbb{T}^2)$, for $s>3$. 
Then, the following estimates hold true.
\begin{equation}
\label{s_integral_estimate_KRM_paper}
\begin{aligned}
\| w^*(t)\|_s^2 &+ \varepsilon^2 (\|m(t)\|_s^2 + \|\xi(t)\|_s^2)+ \|k^*(t)\|_s^2+\|h^*(t)\|_s^2\\
& + \int_0^T  \frac{1}{\varepsilon^2}\|w^*(\theta)\|_s^2 + \|m(\theta)\|_s^2 + \|\xi(\theta)\|_s^2+ \frac{1}{\varepsilon^2}(\|k^*(\theta)\|_s^2+\|h^*(\theta)\|_s^2) \, d\theta\\
&\le c \varepsilon^2 (\|\textbf{u}_0\|_s^2+\|\nabla \textbf{u}_0\|_s^2) \\
& + c(|\rho|_{L_t^\infty L_x^\infty}, \; |\textbf{u}|_{L_t^\infty L_x^\infty}) \int_0^T \| w^*(\theta)\|_s^2 + \varepsilon^2 (\|m(\theta)\|_s^2 + \|\xi(\theta)\|_s^2) \, d\theta \\
& + c(|\rho|_{L_t^\infty L_x^\infty}, \; |\textbf{u}|_{L_t^\infty L_x^\infty}) \int_0^T \|k^*(\theta)\|_s^2+\|h^*(\theta)\|_s^2 \, d\theta, \quad t < T^\varepsilon.
\end{aligned}
\end{equation}
\vspace{3mm}
\begin{equation}
\label{s_energy_estimate}
\begin{aligned}
\| w^*(t)\|_s^2 &+ \varepsilon^2 (\|m(t)\|_s^2 + \|\xi(t)\|_s^2)+\|k^*(t)\|_s^2+\|h^*(t)\|_s^2\\
&\le c \varepsilon^2 (\|\textbf{u}_0\|_s^2+\|\nabla \textbf{u}_0\|_s^2) e^{c(|\rho|_{L_t^\infty L_x^\infty}, \; |\textbf{u}|_{L_t^\infty L_x^\infty})t}, \quad t<T^\varepsilon.
\end{aligned}
\end{equation}
\vspace{3mm}
\begin{equation}
\label{s_1_time_derivative_estimate}
\begin{aligned}
\|\partial_t w^*(t)\|&_{s-1}^2 + \varepsilon^2 (\|\partial_t m(t)\|_{s-1}^2 + \|\partial_t \xi(t)\|_{s-1}^2)+ \|\partial_t k^*(t)\|_{s-1}^2+\|\partial_t h^*(t)\|_{s-1}^2\\
&\le c \varepsilon^2 (\|\textbf{u}_0\|_{s-1}^2+\|\nabla \textbf{u}_0\|_{s-1}^2+\|\nabla^2 \textbf{u}_0\|^2_{s-1}) e^{c(|\rho|_{L_t^\infty L_x^\infty}, \; |\textbf{u}|_{L_t^\infty L_x^\infty})t}, \quad t<T^\varepsilon.
\end{aligned}
\end{equation}
Moreover, there exists $\varepsilon_0, M$ and $T^*<T^\varepsilon$ fixed such that, for $\varepsilon \le \varepsilon_0$, 
\begin{equation}
\label{infty_original_uniform_bounds_KRM_paper}
|\rho \textbf{u}(t)|_\infty \le M, \quad |\rho(t)-\bar{\rho}|_\infty \le \varepsilon M, \quad t \in [0, T^*],
\end{equation}
\begin{equation}
\label{infty_uniform_bounds_KRM_paper}
|\rho(t)|_\infty \le \bar{\rho} + \varepsilon M, \quad |\textbf{u}(t)|_\infty \le \frac{M}{\bar{\rho}+\varepsilon M}, \quad t \in [0, T^*].
\end{equation}
\begin{equation}
\label{L1_Hs_pressure}
\int_0^T |\rho(t)-\bar{\rho}|_\infty \, dt \le  c(M)\varepsilon^2, \quad T \in [0, T^*].
\end{equation}
\end{lemma}
\begin{proof}
We discuss each result separately.
\begin{itemize}
\item Estimate (\ref{s_integral_estimate_KRM_paper}) follows from [\cite{Bianchini3}, Lemma 4.2], the change of variables [\cite{Bianchini3}, (30)] and the Sobolev embedding theorem. Notice that the dependency of $c(|\rho|_{L_t^\infty L_x^\infty}, \cdot)$ on $|\rho|_{L_t^\infty L_x^\infty}$ is a consequence of the quadratic pressure in (\ref{Pressure_approximation_BGK}), see Remark \ref{Remark_previous_setting}, and the estimates of the nonlinear term in [\cite{Bianchini3}, Lemma 4.2].

\item By applying Gronwall's inequality to (\ref{s_integral_estimate_KRM_paper}), one gets (\ref{s_energy_estimate}).

\item Estimate (\ref{s_1_time_derivative_estimate}) follows from [\cite{Bianchini3}, Proposition 3 and (30)]. 

\item For a fixed constant $M>M_0:=\bar{\rho}\|\textbf{u}_0\|_{s+1}$, let us define 
\begin{equation}
\label{T_star_def}
\begin{aligned}
T^*: =\sup_{t \in [0, T^\varepsilon)} &  \Bigg\{ \frac{|\rho(t)-\bar{\rho}|_\infty}{\varepsilon} + |\rho \textbf{u}(t)|_\infty \le M\Bigg\}.
\end{aligned}
\end{equation}

The Sobolev embedding theorem applied to (\ref{s_energy_estimate}) yields, thanks to the definition of $w^*$ in (\ref{translated_variables}), 
\begin{equation}
\frac{|\rho(t)-\bar{\rho}|_\infty}{\varepsilon}+|\rho \textbf{u}(t)|_\infty \le cM_0  e^{c(|\rho|_{L_t^\infty L_x^\infty}, \; |\textbf{u}|_{L_t^\infty L_x^\infty})t}, \qquad t \le T^*.
\end{equation}
The uniform bounds (\ref{infty_original_uniform_bounds_KRM_paper})-(\ref{infty_uniform_bounds_KRM_paper}) are due to the Sobolev embedding theorem applied to (\ref{s_energy_estimate}) and the definition of $T^*$, which depends on $M_0, \, M$. 

\item The last uniform bound is a consequence of the Sobolev embedding theorem applied to (\ref{s_integral_estimate_KRM_paper}), the previous bounds in (\ref{infty_original_uniform_bounds_KRM_paper})-(\ref{infty_uniform_bounds_KRM_paper}), and the definition of $w^*$ in (\ref{translated_variables}).
\end{itemize}
\end{proof}
\subsection{Kinetic entropies and the relative entropy}
Here we recall the definition and the conditions that assure the existence of a kinetic entropy for a discrete velocities BGK model, see \cite{Bouchut} for a detailed discussion.\\
Let $\mathcal{E}$ be a non-empty set of convex entropies for a given limit system. Assume also that $\mathcal{E}$ is separable. A general BGK model under the diffusive scaling reads as follows
\begin{equation}
\label{BGK_general}
\partial_t f_i + \frac{\lambda_i}{\varepsilon}  \cdot \nabla_x f_i = \frac{1}{\varepsilon^2}(M_i(\textbf{u}^\varepsilon)-f_i), \qquad i=1, \cdots, L,
\end{equation}
where $L \ge d$, for $ i=1, \cdots, L,$
\begin{align*}
& f_i(t, x)=(f_i^1, \cdots, f_i^N): \mathbb{R}\times
\mathbb{R}^d \rightarrow \mathbb{R}^N, \\
& \lambda_i=(\lambda_i^1, \cdots, \lambda_i^d), \\
& M_i(\textbf{u}^\varepsilon)=(M_i^1, \cdots, M_i^N): \mathbb{R}^N \rightarrow \mathbb{R}^N.
\end{align*}
and
$\textbf{u}^\varepsilon=\sum_{i=1}^L f_i$ is the approximating vector field, converging to the solution to the limit system, which is established under some consistency conditions, see \cite{Bouchut, CN, Aregba1, Aregba2, BGN} for a detailed discussion. An important feature of these approximations is the existence, under some reasonable conditions, of a kinetic entropy. Set $\mathcal{D}_i:=\{M_i(\textbf{u}): \, \textbf{u} \in \mathcal{U}\}.$
\begin{definition}
\label{kinetic_entropy_general}
A kinetic entropy for system (\ref{BGK_general}) is a convex function $\mathcal{H}(\textbf{f})=\sum_{i=1}^L \mathcal{H}_i(f_i)$, with $\mathcal{H}_i: \mathcal{D}_i \rightarrow \mathbb{R}$, such that, for $\eta(\textbf{u}) \in \mathcal{E}$,
\begin{itemize}
\item (E1) $\mathcal{H}(M(\textbf{u}))=\eta(\textbf{u})$ for every $\textbf{u}\in \mathcal{U}$,
\item (E2) $\mathcal{H}(M(\textbf{u}_f)) \le \mathcal{H}(\textbf{f}),$ where $\textbf{u}_f:=\sum_{i=1}^L f_i \in \mathcal{U}, \, f_i \in \mathcal{D}_i$.
\end{itemize}
\end{definition}
Such a property provides an energy inequality which gives robustness for the scheme, and this is the main advantage of these models with respect to another class of discrete velocities BGK models used in computational physiscs, the Lattice Boltzmann schemes, see \cite{Succi, Wolf}. Indeed, it is easy to see that, multiplying the BGK system (\ref{BGK_general}) by $\nabla_f \mathcal{H}(\textbf{f})$, the minimality (E2)
together with the convexity property, provide the following entropy inequality
\begin{equation}
\label{1entropy inequality}
\partial_t \mathcal{H}(\textbf{f})+{\Lambda} \cdot \nabla_x \mathcal{H}(\textbf{f})=\dfrac{1}{\varepsilon} \nabla_\textbf{f} \mathcal{H}(\textbf{f}) \cdot (M(\textbf{u})-\textbf{f})  \le 0,
\end{equation}
which means that, according with the definition given in \cite{HN}, the kinetic entropy $\mathcal{H}(\textbf{f})$ is dissipative.
More precisely, properties (E1)-(E2) under the hypothesis of  [\cite{Bouchut}, Thm. 2.1] assure that, for any $\eta(\textbf{u}) \in \mathcal{E}$, defining the projector $\mathcal{P}$ such that
\begin{equation}
\label{Projector}
\mathcal{P}\textbf{f}=\sum_{i=1}^L=\textbf{u},
\end{equation}
then
\begin{equation}
\label{minimu_H_precise}
\eta(\textbf{u})=\min_{\mathcal{P}\textbf{f}=\textbf{u}} \mathcal{H}(\textbf{f})=\mathcal{H}(M(\textbf{u})).
\end{equation}
In this context, the Gibbs principle for relaxation and, in particular,  [\cite{Tzavaras}, Prop. 2.1], imply that
\begin{equation}
\label{orthogonality_Tzavaras}
\nabla_{\textbf{f}} \mathcal{H}(M(\textbf{u})) \perp Ker(\mathcal{P}).
\end{equation}
Since $\textbf{f}-{M}(\textbf{u}) \in Ker(\mathcal{P})$, the convexity property of $\mathcal{H}(\textbf{f})$ together with condition (\ref{orthogonality_Tzavaras}) allow us to get the following inequality:
\begin{equation}
\label{dissipative_entropy}
\nabla_\textbf{f} \mathcal{H}(\textbf{f}) \cdot (\textbf{f}-M(\textbf{u})) \le -c |\textbf{f}-M(\textbf{u})|^2, \qquad c=c(|\textbf{f}|_\infty),
\end{equation}
meaning that the kinetic entropy $\mathcal{H}(M(\textbf{u}))$ is strictly dissipative, according to the definition given in \cite{HN}. This dissipative property is the main ingredient to apply the relative entropy method, which provides a uniform bound for the relative entropy. Roughly speaking, the relative entropy can be seen as a perturbation of the kinetic entropy near to the equilibrium represented by the solution to the limit system. A precise definition in the context of hyperbolic relaxation is provided in \cite{Tzavaras}. For diffusive relaxation, we will use the following 
\begin{equation}
\label{relative_entropy_def}
\begin{aligned}
\tilde{\mathcal{H}}(\textbf{f}|\bar{\textbf{f}})&=\mathcal{H}(\textbf{f})-\mathcal{H}(\overline{\mathcal{M}}(\bar{w}))-\nabla_\textbf{f}\mathcal{H}(\overline{\mathcal{M}}(\bar{w}))\cdot (\textbf{f}-\overline{\mathcal{M}}(\bar{w}))\\
&=\sum_i \mathcal{H}_i(f_i) - \mathcal{H}_i(\overline{\mathcal{M}}_i(\bar{w}))-\nabla_{f_i} \mathcal{H}_i(\overline{\mathcal{M}}_i(\bar{w})) \cdot (f_i-\overline{\mathcal{M}}_i(\bar{w})),
\end{aligned}
\end{equation}
where $\mathcal{H}(\textbf{f})$ is in Definition \ref{kinetic_entropy_general}, and $\overline{\mathcal{M}}(\bar{w})=(\overline{\mathcal{M}}_i(\bar{w}))_{i=1, \cdots, 5}$ are the perturbed Maxwellians in (\ref{Perturbed_Maxwellians}) evaluated in the solution $\bar{w}=(\bar{\rho}, \varepsilon \bar{\rho}\bar{\textbf{u}})$ to the incompressible Navier-Stokes equations (\ref{real_NS}).
\subsection{Quantifying the dissipation}
The aim of this part is to characterize and to quantify the dissipative terms resulting from the relative entropy estimate.
Hereafter, we will drop the apex $\varepsilon$ when there is no ambiguity. We start with two preliminary lemmas.
\begin{lemma}
\label{lemma_derivatives-BGK}
Let $\eta(w)$ be defined in (\ref{entropy_isentropic_Euler}). Let 
$\mathcal{H}(\textbf{f})=\sum_{i=1}^5 \mathcal{H}_i(f_i)$
be a kinetic entropy associated with the vector-BGK model in (\ref{BGK_NS}), such that $\mathcal{H}(M(w))=\eta(w)$.
Then the following entropy expansion is satisfied:
\begin{equation*}
\begin{aligned}
\frac{1}{\varepsilon^2} \int_0^T \iint & \sum_{i=1}^5 \nabla_{f_i} \mathcal{H}_i(f_i) \cdot  (M_i-f_i) \; dt \, dx \, dy\\
&=-\int_0^T \iint \Bigg[ \frac{\nabla^2_{w}\mathcal{\eta}(w)}{2a \lambda^2 \tau}\cdot (m-\frac{A_1(w)}{\varepsilon})\Bigg]\cdot (m-\frac{A_1(w)}{\varepsilon})\; dt \, dx \, dy\\
&-\int_0^T \iint \Bigg[ \frac{\nabla^2_{w}\mathcal{\eta}(w)}{2a \lambda^2 \tau}\cdot (\xi-\frac{A_2(w)}{\varepsilon})\Bigg] \cdot (\xi-\frac{A_2(w)}{\varepsilon})\; dt \, dx \, dy\\
& -\int_0^T \int \int \Bigg[ \frac{\nabla^2_{w}\mathcal{\eta}(w)}{2a \varepsilon^2 \tau}\cdot (k-2aw) \Bigg] \cdot (k-2aw) \; dt \, dx \, dy\\
& -\int_0^T \iint  \Bigg[ \frac{\nabla^2_{w}\mathcal{\eta}(w)}{2a \varepsilon^2 \tau}\cdot (h-2aw) \Bigg] \cdot (h-2aw) \; dt \, dx \, dy\\
&-\int_0^T \iint  \Bigg[\frac{\nabla^2_w \eta(w)}{(1-4a)\tau \varepsilon^2}\cdot (4aw-(k+h))\Bigg] \cdot (4aw-(k+h)) \; dt \, dx \, dy \\
&+ O(\varepsilon^3).
\end{aligned}
\end{equation*}
\end{lemma}
\begin{proof}
First of all, the uniform bounds (\ref{infty_uniform_bounds_KRM_paper}) and [\cite{Bouchut}, Theorem 2.1] provide the existence of a kinetic entropy for (\ref{BGK_NS}), such that
$$\mathcal{H}({M}(w))=\eta(w) \; \text{in } (\ref{entropy_isentropic_Euler}), \quad \nabla_{f_i}\mathcal{H}_i(M_i(w))=\nabla_{w}\eta(w), \quad i=1, \cdots, 5.$$
We point out that the spectrum of the Jacobian matrices of the Maxwellians in (\ref{Maxwellians}) is positive provided that the parameter $a$ in the expressions (\ref{Maxwellians}) is positive and $\lambda>0$ is big enough (Assumptions \ref{assumptions_a}). This remark, together with the bounds in (\ref{infty_uniform_bounds_KRM_paper}), assure the existence of a kinetic convex and dissipative entropy for our system, thanks to [\cite{Bouchut}, Theorem 2.1]. 
Notice that in the course of our computations, the densities $\textbf{f}^{\, \varepsilon}$ remain in a bounded set, close enough to the hyperbolic equilibrium.\\
Now we consider the following expansion
\begin{equation}
\label{entropy_expansion}
\begin{aligned}
\frac{1}{\varepsilon^2}\sum_{i=1}^5 & \nabla_{f_i} \mathcal{H}_i(f_i)  \cdot (M_i-f_i)\\
&=\frac{1}{\varepsilon^2}\sum_{i=1}^5 \nabla_{f_i} \mathcal{H}_i(M_i)\cdot (f_i-M_i) + \frac{1}{\varepsilon^2}\sum_{i=1}^5 \nabla^2_{f_i} \mathcal{H}_i(M_i) \cdot (f_i-M_i)\cdot (f_i-M_i)\\
& + O(\frac{|f_i-M_i|^3}{\varepsilon^2})\\
&=-\frac{1}{\varepsilon^2}\sum_{i=1}^5 \nabla^2_{f_i} \mathcal{H}_i(M_i) \cdot (f_i-M_i) \cdot (f_i-M_i)\\
&+ O(\frac{|f_i-M_i|^3}{\varepsilon^2}).
\end{aligned}
\end{equation}
where the first term vanishes thanks to the orthogonality property [\cite{Tzavaras}, Proposition 2.1].
For $i=1, \cdots, 4$, the first term of the last equality reads
\begin{align*}
&-\frac{1}{\varepsilon^2}\int_0^T \iint \nabla^2_{f_i}\mathcal{H}_i(M_i) \cdot (f_i-M_i) \cdot (f_i-M_i) \; dt \, dx \, dy\\
&=-\frac{1}{\varepsilon^2}\int_0^T \iint \nabla^2  \mathcal{H}_i(aw\pm \frac{A_i(w)}{2\lambda}) \cdot (f_i-M_i) \cdot (f_i-M_i)\; dt \, dx \, dy.
\end{align*}
Note that, from (\ref{w_def})-(\ref{Fluxes_BGK}) and Lemma \ref{lemma_energy_estimates_KRM_paper},
$$w=\left(\begin{array}{c}
\rho\\
\varepsilon \rho u_1\\
\varepsilon \rho u_2
\end{array}\right)
=\left(\begin{array}{c}
O(1)\\
O(\varepsilon)\\
O(\varepsilon)
\end{array}\right), 
$$
\begin{align*}
A_{1}(w)=\left(\begin{array}{c}
\varepsilon \rho u_1\\
\varepsilon^2 \rho u_1^2+\frac{\rho^2 - \bar{\rho}^2}{2\bar{\rho}}\\
\varepsilon^2 \rho u_1 u_2
\end{array}\right)
=\left(\begin{array}{c}
O(\varepsilon)\\
O(\varepsilon^2)\\
O(\varepsilon^2)
\end{array}\right),
\end{align*}
\begin{align*}
A_{2}(w)=\left(\begin{array}{c}
\varepsilon \rho u_1\\
\varepsilon^2 \rho u_1 u_2\\
\varepsilon^2 \rho u_2^2+\frac{\rho^2 - \bar{\rho}^2}{2\bar{\rho}}\\
\end{array}\right)
=\left(\begin{array}{c}
O(\varepsilon)\\
O(\varepsilon^2)\\
O(\varepsilon^2)
\end{array}\right).
\end{align*}
This way,
\begin{align*}
\nabla_{f_i} \mathcal{H}(M_i(w))=\nabla_{f_i} \mathcal{H}\Bigg(aw\pm \frac{A_i(w)}{2\lambda}\Bigg)=\nabla_{f_i}\mathcal{H}(aw)+O(\varepsilon).
\end{align*}
Moreover, from \cite{Bouchut}, it is also known that
$$\nabla_{f_i} \mathcal{H}(M_i(w))=\nabla_w \eta(w).$$
Differentiating again the previous equivalent expressions, 
\begin{equation}
\label{real_entropy_expansion}
\nabla^2_{f_i}   \mathcal{H}_i(aw)=\frac{1}{a}\nabla^2_w \eta(w) + O(\varepsilon).
\end{equation}
Thus, the last equality yields
\begin{align*}
& -\frac{1}{\varepsilon^2}\int_0^T \iint \nabla^2_{f_i}   \mathcal{H}_i(aw \pm \frac{A_i(w)}{2 \lambda}) \cdot (f_i-M_i)\cdot (f_i-M_i)\; dt \, dx \, dy \\
& \le -\frac{1}{\varepsilon^2}\int_0^T \iint \frac{1}{a} \nabla^2_w \eta(w) \cdot (f_i-M_i)\cdot (f_i-M_i) \; dt \, dx \, dy \\
&+\frac{c(|w|_{L_t^\infty L_x^\infty})}{\varepsilon} |f_i-M_i|^2_{L_t^\infty L_x^\infty}.
\end{align*}

Now, from  (\ref{BGK_NS})-(\ref{variables_BGK}),
\begin{equation}
\label{estimates_lemma_BGK}
\begin{aligned}
& \frac{M_1-f_1}{\varepsilon^2}=\partial_t f_1 + \frac{\lambda}{\varepsilon} \partial_x f_1=\frac{1}{2}(\partial_t k + \partial_x m) + \frac{1}{2 \lambda \varepsilon} (\varepsilon^2 \partial_t m + \lambda^2 \partial_x k), \\
&  \frac{M_3-f_3}{\varepsilon^2}=\partial_t f_3 - \frac{\lambda}{\varepsilon} \partial_x f_3=\frac{1}{2}(\partial_t k + \partial_x m) - \frac{1}{2 \lambda \varepsilon} (\varepsilon^2 \partial_t m + \lambda^2 \partial_x k), \\
& \frac{M_2-f_2}{\varepsilon^2}=\partial_t f_2 + \frac{\lambda}{\varepsilon} \partial_y f_2 = \frac{1}{2} (\partial_t h + \partial_y \xi) + \frac{1}{2 \lambda \varepsilon} (\varepsilon^2 \partial_t \xi + \lambda^2 \partial_y h), \\
& \frac{M_4-f_4}{\varepsilon^2}=\partial_t f_4 - \frac{\lambda}{\varepsilon} \partial_y f_4 = \frac{1}{2} (\partial_t h + \partial_y \xi) - \frac{1}{2 \lambda \varepsilon} (\varepsilon^2 \partial_t \xi + \lambda^2 \partial_y h).
\end{aligned}
\end{equation}
Lemma \ref{lemma_energy_estimates_KRM_paper} and the previous equalities imply that
$$\frac{c(|w|_{L_t^\infty L_x^\infty})}{\varepsilon} |f_i-M_i|^2_{L_t^\infty L_x^\infty}=O(\varepsilon^3),$$
and so, by using the change of variables (\ref{variables_BGK}),
\begin{align*}
&-\frac{1}{\varepsilon^2} \int_0^T \iint \sum_{i=1}^4 \nabla_{f_i} \mathcal{H}_i(f_i) \cdot  (M_i-f_i) \; dt \, dx \, dy\\
&=-\frac{1}{\varepsilon^2} \int_0^T \iint \sum_{i=1}^4 \frac{1}{a} \nabla^2_w \eta(w) \cdot  (M_i-f_i) \cdot  (M_i-f_i) \; dt \, dx \, dy+O(\varepsilon^3)\\
&=\int_0^T \iint \Bigg[ \frac{\nabla^2_{w}\mathcal{\eta}(w)}{2a \lambda^2 \tau}\cdot (m-\frac{A_1(w)}{\varepsilon})\Bigg]\cdot (m-\frac{A_1(w)}{\varepsilon})\; dt \, dx \, dy\\
&-\int_0^T \iint \Bigg[ \frac{\nabla^2_{w}\mathcal{\eta}(w)}{2a \lambda^2 \tau}\cdot (\xi-\frac{A_2(w)}{\varepsilon})\Bigg] \cdot (\xi-\frac{A_2(w)}{\varepsilon})\; dt \, dx \, dy\\
& -\int_0^T \int \int \Bigg[ \frac{\nabla^2_{w}\mathcal{\eta}(w)}{2a \varepsilon^2 \tau}\cdot (k-2aw) \Bigg] \cdot (k-2aw) \; dt \, dx \, dy\\
& -\int_0^T \iint  \Bigg[ \frac{\nabla^2_{w}\mathcal{\eta}(w)}{2a \varepsilon^2 \tau}\cdot (h-2aw) \Bigg] \cdot (h-2aw) \; dt \, dx \, dy+O(\varepsilon^3).
\end{align*}
The expansion
\begin{align*}
& -\frac{1}{\varepsilon^2} \int_0^T \iint \mathcal{H}_i(f_5) \cdot  (M_5-f_5) \; dt \, dx \, dy \\
&=-\int_0^T \iint  \Bigg[\frac{\nabla^2_w \eta(w)}{(1-4a)\tau \varepsilon^2}\cdot (4aw-(k+h))\Bigg] \cdot (4aw-(k+h)) \; dt \, dx \, dy + O(\varepsilon^3)
\end{align*}
is obtained in analogous way.
\end{proof}

\begin{lemma}
\label{corollary_perturbed_Maxwellians}
Consider the limit solution $\overline{\mathcal{M}}_i$, for $i=1, \cdots, 4$, in (\ref{Perturbed_Maxwellians}). Then
\begin{equation}
\label{Perturbed_Maxwellians_expansion}
\begin{aligned}
\nabla_{f_i} \overline{\mathcal{H}}_i(\overline{\mathcal{M}}_i)&=\nabla_{f_i} \mathcal{H}_i(\overline{{M}}_i) \mp a \varepsilon \lambda \tau \nabla^2_{f_i} \mathcal{H}_i(\overline{M}_i)\partial_{x_j} \bar{w}+O(\varepsilon^3)\\
&=\nabla_w \eta(\bar{w})\mp \lambda \varepsilon \tau \nabla^2_w\eta(\bar{w})\partial_{x_j}\bar{w}+O(\varepsilon^3), \quad j=1, 2.
\end{aligned}
\end{equation}
\end{lemma}
\begin{proof}
The proof follows by Taylor expansions and (\ref{real_entropy_expansion}), in the spirit of Lemma \ref{lemma_energy_estimates_KRM_paper}.
\end{proof}

\section{Relative entropy estimate for the vector-BGK model}
\label{REsection}
Our main result is stated here.
\begin{theorem}
\label{Theorem_entropy_inequality_NS}
Consider the vector-BGK model in (\ref{BGK_NS}) for the two-dimensional incompressible Navier-Stokes equations in (\ref{real_NS}) on $[0, +\infty) \times \mathbb{T}^2$, endowed with a kinetic entropy $\mathcal{H}(\textbf{f}^{\; \varepsilon})$, whose existence and properties are given by Lemma \ref{lemma_derivatives-BGK}. Let $\bar{\textbf{u}}=(\bar{u}_1, \bar{u}_2)$, $\nabla \bar{P}$ be a smooth velocity field and pressure satisfying the incompressible Navier-Stokes equations (\ref{real_NS}) on $[0, +\infty) \times \mathbb{T}^2$ and $\{\textbf{f}^{\; \varepsilon}\}$ be a family of smooth solutions to (\ref{BGK_NS}) and emanating from smooth initial data $\textbf{u}_0$ in (\ref{real_NS_initial_data}) and $\textbf{f}_0=(f_i(0, x))_{i=1, \cdots, 5}$ in (\ref{initial_conditions_BGK}). Then, defining $w^\varepsilon=\sum_i f_i^\varepsilon=(\rho^\varepsilon, \varepsilon \rho^\varepsilon \textbf{u}^\varepsilon)$, the following estimate holds for any $T>0$ and for $\varepsilon \le \varepsilon_0$, where $\varepsilon_0$ is fixed and it depends on $M_0=\bar{\rho}\|\textbf{u}_0\|_{s+1},$ 
$$\sup_{t \in [0, T]} \frac{\|\rho(t)-\bar{\rho}\|_{s'}}{\varepsilon} + \|\textbf{u}(t)-\bar{\textbf{u}}(t)\|_{s'} \le c \varepsilon^{\frac{1}{2}-\delta},$$
with $s>3, \; 0<s'<s$ and $\delta:=\frac{s-s'}{2s}.$ 
Moreover, for $\varepsilon \le \varepsilon_0$, the solutions $(\rho^\varepsilon, \textbf{u}^\varepsilon)$ to the approximating system (\ref{BGK_NS}) are globally bounded in time, and for $\varepsilon \rightarrow 0$,
$$\frac{\nabla ((\rho^\varepsilon)^2-\bar{\rho}^2)}{\varepsilon^2} {\rightharpoonup^\star}\nabla \bar{P} \quad \text{in } L_t^\infty H_x^{s-3}.$$
\end{theorem}

The global in time convergence proof is based on the use of the relative entropy inequality, which is stated here.

\begin{theorem}
\label{Theorem_L2_estimate}
Under the hypothesis of Theorem \ref{Theorem_entropy_inequality_NS}, let $T^*$ be defined in (\ref{T_star_def}). Then the relative entropy method provides the following estimate:

$$\sup_{t \in [0, T^*]} \frac{\|\rho(t)-\bar{\rho}\|_{0}}{\varepsilon} + \|\rho\textbf{u}(t)-\bar{\rho}\bar{\textbf{u}}(t)\|_{0} \le c \sqrt{\varepsilon}.$$
\end{theorem}

\begin{proof}
We start by recalling the definition of the relative entropy in (\ref{relative_entropy_def}),
\begin{align*}
\tilde{\mathcal{H}}(\textbf{f}|\bar{\textbf{f}})&=\mathcal{H}(\textbf{f})-\mathcal{H}(\overline{\mathcal{M}})-\nabla_\textbf{f}\mathcal{H}(\overline{\mathcal{M}})\cdot (\textbf{f}-\overline{\mathcal{M}})\\
&=\sum_i \mathcal{H}_i(f_i) - \mathcal{H}_i(\overline{\mathcal{M}}_i)-\nabla_{f_i} \mathcal{H}_i(\overline{\mathcal{M}}_i) \cdot (f_i-\overline{\mathcal{M}}_i),
\end{align*}
where the limit solutions $\overline{\mathcal{M}}_i=\overline{\mathcal{M}}_i (\bar{\rho}, \varepsilon \bar{\rho} \bar{\textbf{u}}), \; i=1, \cdots, 5,$ are in (\ref{Perturbed_Maxwellians}), $\bar{\rho}$ is a constant density, $\bar{\textbf{u}}$ is the smooth solution to (\ref{real_NS}),
and the associated entropy-flux is given by
\begin{align*}
\tilde{Q}(\textbf{f}|\bar{\textbf{f}})&=\frac{\lambda}{\varepsilon}\left(\begin{array}{c}
\mathcal{H}_1(f_1)-\mathcal{H}_3(f_3)-(\mathcal{H}_1(\overline{\mathcal{M}}_1)-\mathcal{H}_3(\overline{\mathcal{M}}_3))\\\\
\mathcal{H}_2(f_2)-\mathcal{H}_4(f_4)-(\mathcal{H}_2(\overline{\mathcal{M}}_2)-\mathcal{H}_4(\overline{\mathcal{M}}_4))
\end{array}\right)\\\\
&-\frac{\lambda}{\varepsilon}\left(\begin{array}{c} \nabla_{f_1} \mathcal{H}_1(\overline{\mathcal{M}}_1) (f_1-\overline{\mathcal{M}}_1)-\nabla_{f_3} \mathcal{H}_3(\overline{\mathcal{M}}_3) (f_3-\overline{\mathcal{M}}_3)\\\\
\nabla_{f_2}\mathcal{H}_2(\overline{\mathcal{M}}_2) (f_2-\overline{\mathcal{M}}_2)-\nabla_{f_4} \mathcal{H}_4(\overline{\mathcal{M}}_4) (f_4-\overline{\mathcal{M}}_4)\end{array}\right).
\end{align*}
Hereafter, we adopt the following notation, $\overline{\mathcal{H}}_i:=\mathcal{H}_i(\overline{\mathcal{M}}_i)$.
Now we proceed to get the desired inequality.
\begin{align*}
&\int_0^T \iint \partial_t \tilde{\mathcal{H}}(\textbf{f}|\bar{\textbf{f}}) + \nabla_x \cdot \tilde{Q}(\textbf{f}|\bar{\textbf{f}}) \; dt \, dx \, dy \\
&=\int_0^T \iint \partial_t \mathcal{H}(\textbf{f})+\frac{\lambda}{\varepsilon}\partial_x ( \mathcal{H}_1(f_1)- \mathcal{H}_3(f_3)) + \frac{\lambda}{\varepsilon}\partial_y ( \mathcal{H}_2(f_2)- \mathcal{H}_4(f_4)) \; dt \, dx \, dy\\\\
&-\int_0^T \iint  \partial_t  \mathcal{H}(\overline{\mathcal{M}})+ \frac{\lambda}{\varepsilon}\partial_x (\mathcal{H}_1(\overline{\mathcal{M}}_1)-\mathcal{H}_3(\overline{\mathcal{M}}_3)) \; dt \, dx \, dy\\
& - \int_0^T \iint \frac{\lambda}{\varepsilon}\partial_y (\mathcal{H}_2(\overline{\mathcal{M}}_2)-\mathcal{H}_4(\overline{\mathcal{M}}_4)) \; dt \, dx \, dy \\\\
&\begin{aligned}-\int_0^T \iint \partial_t&(\nabla_{f_1}\mathcal{H}_1(\overline{\mathcal{M}}_1) (f_1-\overline{\mathcal{M}}_1)+\nabla_{f_2}\mathcal{H}_2(\overline{\mathcal{M}}_2) (f_2-\overline{\mathcal{M}}_2)\\
&+\nabla_{f_3}\mathcal{H}_3(\overline{\mathcal{M}}_3) (f_3-\overline{\mathcal{M}}_3)+\nabla_{f_4}\mathcal{H}_4(\overline{\mathcal{M}}_4) (f_4-\overline{\mathcal{M}}_4)\\
&+\nabla_{f_5}\mathcal{H}_5(\overline{\mathcal{M}}_5) (f_5-\overline{\mathcal{M}}_5)) \; dt \, dx \, dy\end{aligned} \\\\
&-\int_0^T \iint\frac{\lambda}{\varepsilon}\partial_x(\nabla_{f_1}\mathcal{H}_1(\overline{\mathcal{M}}_1) (f_1-\overline{\mathcal{M}}_1)-\nabla_{f_3}\mathcal{H}_3(\overline{\mathcal{M}}_3) (f_3-\overline{\mathcal{M}}_3)) \; dt \, dx \, dy\\\\
&-\int_0^T \iint\frac{\lambda}{\varepsilon}\partial_y(\nabla_{f_2}\mathcal{H}_2(\overline{\mathcal{M}}_2) (f_2-\overline{\mathcal{M}}_2)-\nabla_{f_4}\mathcal{H}_4(\overline{\mathcal{M}}_4) (f_4-\overline{\mathcal{M}}_4)) \; dt \, dx \, dy\\\\
&=I_1+I_2+I_3+I_4.
\end{align*}
First of all, $I_1$ is already estimated in Lemma \ref{lemma_energy_estimates_KRM_paper}. 
Now, let us consider $I_2$. \\The following expansions are based on Lemma \ref{corollary_perturbed_Maxwellians}.
\begin{align*}
&-\int_0^T \iint\partial_t (\overline{\mathcal{H}}_1+\overline{\mathcal{H}}_2+\overline{\mathcal{H}}_3+\overline{\mathcal{H}}_4+\overline{\mathcal{H}}_5) \; dt \, dx \, dy\\
&-\int_0^T \iint\frac{\lambda}{\varepsilon}\partial_x (\overline{\mathcal{H}}_1-\overline{\mathcal{H}}_3)+\frac{\lambda}{\varepsilon}\partial_y(\overline{\mathcal{H}}_2-\overline{\mathcal{H}}_4) \; dt \, dx \, dy\\\\
&=-\int_0^T \iint (\nabla_w \eta(\bar{w})-a\varepsilon \lambda \tau \nabla^2_{f_1} \overline{\mathcal{H}}_1\partial_x \bar{w}) (\partial_t \overline{\mathcal{M}}_1+\frac{\lambda}{\varepsilon}\partial_x \overline{\mathcal{M}}_1)  \; dt \, dx \, dy\\
&-\int_0^T \iint (\nabla_w \eta(\bar{w})+a\varepsilon \lambda \tau  \nabla^2_{f_3} \overline{\mathcal{H}}_3\partial_x \bar{w}) (\partial_t \overline{\mathcal{M}}_3-\frac{\lambda}{\varepsilon}\partial_x \overline{\mathcal{M}}_3)  \; dt \, dx \, dy\\
&-\int_0^T \iint (\nabla_w \eta(\bar{w})-a\varepsilon \lambda \tau  \nabla^2_{f_2} \overline{\mathcal{H}}_2 \partial_y \bar{w}) (\partial_t \overline{\mathcal{M}}_2+\frac{\lambda}{\varepsilon}\partial_y \overline{\mathcal{M}}_2)  \; dt \, dx \, dy\\
&-\int_0^T \iint (\nabla_w \eta(\bar{w})+a\varepsilon \lambda \tau  \nabla^2_{f_4} \overline{\mathcal{H}}_4 \partial_y \bar{w}) (\partial_t \overline{\mathcal{M}}_4-\frac{\lambda}{\varepsilon} \overline{\mathcal{M}}_4)  \; dt \, dx \, dy\\
&-\int_0^T \iint\nabla_{w}\eta(\bar{w}) \cdot \partial_t \overline{\mathcal{M}}_5 \; dt \, dx \, dy\\
& -\int_0^T \iint \nabla_{f_1}\overline{\mathcal{H}}_1(\partial_t \overline{\mathcal{M}}_1+\frac{\lambda}{\varepsilon}\partial_x \overline{\mathcal{M}}_1)-\nabla_{f_3}\overline{\mathcal{H}}_3(\partial_t \overline{\mathcal{M}}_3-\frac{\lambda}{\varepsilon}\partial_x \overline{\mathcal{M}}_3) \; dt \, dx \, dy\\
&-\int_0^T \iint \nabla_{f_2}\overline{\mathcal{H}}_2(\partial_t \overline{\mathcal{M}}_2+\frac{\lambda}{\varepsilon}\partial_y \overline{\mathcal{M}}_2)-\nabla_{f_4}\overline{\mathcal{H}}_4(\partial_t \overline{\mathcal{M}}_4-\frac{\lambda}{\varepsilon}\partial_y \overline{\mathcal{M}}_4)\; dt \, dx \, dy+O(\varepsilon^3) \\\\
&=-\int_0^T \iint \nabla_w \eta(\bar{w}) \cdot \partial_t(\overline{\mathcal{M}}_1+\overline{\mathcal{M}}_2+\overline{\mathcal{M}}_3+\overline{\mathcal{M}}_4+\overline{\mathcal{M}}_5) \; dt \, dx \, dy \\
&-\int_0^T \iint \nabla_{w}\eta(\bar{w}) [\frac{\lambda}{\varepsilon}\partial_x (\overline{\mathcal{M}}_1-\overline{\mathcal{M}}_3)+\frac{\lambda}{\varepsilon}\partial_y (\overline{\mathcal{M}}_2-\overline{\mathcal{M}}_4)]  \; dt \, dx \, dy\\
&+ 2a \tau \lambda^2\int_0^T \iint  (\nabla^2_w \eta(\bar{w}) \cdot \partial_x \bar{w}) \partial_x \bar{w}+(\nabla^2_w \eta(\bar{w}) \cdot \partial_y \bar{w})\partial_y \bar{w}\; dt \, dx \, dy+O(\varepsilon^3) \\\\
&= - \int_0^T \iint \nabla_w \eta(\bar{w}) [\partial_t \bar{w}+\partial_x \frac{A_1(\bar{w})}{\varepsilon}+\partial_y \frac{A_2(\bar{w})}{\varepsilon}-\nu (\partial_{xx} \bar{w}+\partial_{yy} \bar{w})]  \; dt \, dx \, dy\\
&+ \int_0^T \iint \tau\Bigg[\frac{\nabla^2_w \eta(\bar{w})}{2a \lambda^2} \cdot (\varepsilon^2 \partial_t \bar{m}+\lambda^2 \partial_x \bar{k})\Bigg]\cdot (\varepsilon^2 \partial_t \bar{m}+\lambda^2 \partial_x \bar{k})  \; dt \, dx \, dy\\
&+ \int_0^T \iint \tau\Bigg[\frac{\nabla^2_w \eta(\bar{w})}{2a \lambda^2} \cdot (\varepsilon^2 \partial_t \bar{\xi}+\lambda^2 \partial_y \bar{h})\Bigg]\cdot (\varepsilon^2 \partial_t \bar{\xi}+\lambda^2 \partial_y \bar{h})  \; dt \, dx \, dy+O(\varepsilon^3).
\end{align*}
\begin{remark}
Notice that the last equalities follow by adding and subtracting terms of order $(\varepsilon^2 \partial_t \bar{m}) \cdot \partial_x \bar{k}, \; (\varepsilon^2 \partial_t \bar{\xi}) \cdot \partial_y \bar{h}, \;  |\varepsilon^2 \partial_t \bar{m}|^2, \; |\varepsilon^2 \partial_t \bar{\xi}|^2$, where
\begin{align*}
& \partial_x \bar{k}=2a \partial_x \bar{w}=2a\partial_x\left(\begin{array}{c}
\bar{\rho}\\
\varepsilon \bar{\rho} \bar{u}_1\\
\varepsilon \bar{\rho} \bar{u}_2\\
\end{array}\right)=O(\varepsilon), \\
& \partial_y \bar{k}=2a \partial_y \bar{w}=O(\varepsilon), \\
&\varepsilon^2 \partial_t \bar{m}=\varepsilon^2\partial_t [\frac{A_1(\bar{w})}{\varepsilon}-\nu \partial_x \bar{w}]=\varepsilon^2 \partial_t \Bigg[\left(\begin{array}{c}
\bar{\rho}\bar{u}_1 \\
\varepsilon \bar{\rho} \bar{u}_1 + \varepsilon \bar{P}\\
\varepsilon \bar{\rho} \bar{u}_1 \bar{u}_2
\end{array}\right)-\nu \partial_x\bar{w}\Bigg]=O(\varepsilon^2),\\
&\varepsilon^2 \partial_t \bar{\xi}=\varepsilon^2\partial_t [\frac{A_2(\bar{w})}{\varepsilon}-\nu \partial_y \bar{w}]=\varepsilon^2 \partial_t \Bigg[\left(\begin{array}{c}
\bar{\rho}\bar{u}_2 \\
\varepsilon \bar{\rho} \bar{u}_1 \bar{u}_2 \\
\varepsilon \bar{\rho} \bar{u}_2^2+ \varepsilon \bar{P}
\end{array}\right)-\nu \partial_y\bar{w}\Bigg]=O(\varepsilon^2).\\
\end{align*}
This way, every remainder term is $O(\varepsilon^3)$.
\end{remark}
Next, we consider $I_3$.
\begin{align*}
I_3&=- \int_0^T \iint \partial_t [\nabla_{f_1}\mathcal{H}_1(\overline{\mathcal{M}}_1) (f_1-\overline{\mathcal{M}}_1) + \nabla_{f_2}\mathcal{H}_2(\overline{\mathcal{M}}_2) (f_2-\overline{\mathcal{M}}_2)] \; dt \, dx \, dy\\
& - \int_0^T \iint \partial_t [\nabla_{f_3}\mathcal{H}_3(\overline{\mathcal{M}}_3) (f_3-\overline{\mathcal{M}}_3) 
+  \nabla_{f_4}\mathcal{H}_4(\overline{\mathcal{M}}_4) (f_4-\overline{\mathcal{M}}_4)] \; dt \, dx \, dy\\
&- \int_0^T \iint  \partial_t [\nabla_{f_5}\mathcal{H}_5(\overline{\mathcal{M}}_5) (f_5-\overline{\mathcal{M}}_5)]  \; dt \, dx \, dy\\\\
& = -\int_0^T \iint\partial_t[\nabla_w \eta(\bar{w})\cdot (w-\bar{w})] \; dt \, dx \, dy\\
& +\varepsilon\lambda \tau \int_0^T \iint \partial_t[\nabla^2_w\eta(\bar{w}) \cdot \partial_x \bar{w} \cdot (f_1-f_3-(\overline{\mathcal{M}}_1-\overline{\mathcal{M}}_3))] \; dt \, dx \, dy \\
&+\varepsilon \lambda \tau \int_0^T \iint \partial_t[\nabla^2\eta(\bar{w}) \cdot \partial_y \bar{w} \cdot (f_2-f_4-(\overline{\mathcal{M}}_2-\overline{\mathcal{M}}_4))] \; dt \, dx \, dy + O(\varepsilon^3)\\\\
&= -\int_0^T \iint\partial_t[\nabla_w \eta(\bar{w})\cdot (w-\bar{w})] \; dt \, dx \, dy\\
& + \varepsilon^2 \tau\int_0^T \iint \partial_t [\nabla^2_w\eta(\bar{w}) \cdot \partial_x \bar{w}\cdot (m-\frac{A_1(\bar{w})}{\varepsilon}+2a\lambda^2 \tau \partial_x \bar{w})] \; dt \, dx \, dy\\
& + \varepsilon^2 \tau \int_0^T \iint \partial_t [\nabla^2_w \eta(\bar{w}) \cdot \partial_y \bar{w} \cdot (\xi - \frac{A_2(\bar{w})}{\varepsilon}+2a\lambda^2 \tau \partial_y \bar{w})] \; dt \, dx \, dy + O(\varepsilon^3)
\end{align*}
\begin{align*}
&=\int_0^T \iint \Bigg[ \nabla^2_w \eta(\bar{w})\cdot [\partial_x \frac{A_1(\bar{w})}{\varepsilon}+\partial_y \frac{A_2(\bar{w})}{\varepsilon}-2a\lambda^2 \tau \partial_{xx} \bar{w}-2a \lambda^2 \tau \partial_{yy} \bar{w}]\cdot  (w-\bar{w}) \\
&\qquad \qquad \; + {\nabla_w \eta(\bar{w})\cdot [\partial_x m + \partial_y \xi - \partial_x \frac{A_1(\bar{w})}{\varepsilon}-\partial_y \frac{A_2(\bar{w})}{\varepsilon}+2a\lambda^2 \tau \partial_{xx} \bar{w} + 2a \lambda^2 \partial_{yy} \bar{w}]}\\
&\qquad \qquad \; + \tau { \nabla^2_w \eta(\bar{w}) \cdot  \partial_x \bar{w} \cdot  (\varepsilon^2 \partial_t m + \lambda^2 \partial_x k)} - {\tau \lambda^2 \nabla^2_w \eta(\bar{w}) \cdot \partial_x \bar{w} \cdot \partial_x k}\\
&\qquad \qquad \; +{\tau \nabla^2_w \eta(\bar{w}) \cdot \partial_y \bar{w} \cdot (\varepsilon^2 \partial_t \xi + \lambda^2 \partial_y h)} - {\tau \lambda^2 \nabla^2_w \eta(\bar{w})\cdot \partial_y \bar{w}\cdot \partial_y h}\Bigg] \; dt \, dx \, dy\\
&\qquad \qquad \;+O(\varepsilon^3).
\end{align*}
It remains to deal with the last term.
\begin{align*}
I_4&=-\int_0^T \iint\frac{\lambda}{\varepsilon}\partial_x [\nabla_{f_1}\mathcal{H}_1(\overline{\mathcal{M}}_1)(f_1-\overline{\mathcal{M}}_1) - \nabla_{f_3}\mathcal{H}_3(\overline{\mathcal{M}}_3)(f_3-\overline{\mathcal{M}}_3)] \; dt \, dx \, dy\\
& -\int_0^T \iint\frac{\lambda}{\varepsilon}\partial_y [\nabla_{f_2}\mathcal{H}_2(\overline{\mathcal{M}}_2)(f_2-\overline{\mathcal{M}}_2) - \nabla_{f_4}\mathcal{H}_4(\overline{\mathcal{M}}_4)(f_4-\overline{\mathcal{M}}_4)] \; dt \, dx \, dy\\
&=-\int_0^T \iint\frac{\lambda}{\varepsilon} \partial_x \Bigg[ (\nabla_w \eta(\bar{w})-\varepsilon \lambda \tau  \nabla^2_w \eta(\bar{w})\partial_x \bar{w})(f_1-\overline{\mathcal{M}}_1) \\
&\qquad - (\nabla_w \eta(\bar{w})+\varepsilon \lambda \tau \nabla^2_w \eta(\bar{w})\partial_x \bar{w})(f_3-\overline{\mathcal{M}}_3)\Bigg]\; dt \, dx \, dy\\
& -\int_0^T \iint\frac{\lambda}{\varepsilon}\partial_y \Bigg[ (\nabla_w \eta(\bar{w})-\varepsilon \lambda \tau \nabla^2_w \eta(\bar{w})\partial_y \bar{w} )(f_2-\overline{\mathcal{M}}_2) \\
&\qquad - (\nabla_w \eta(\bar{w})+\varepsilon \lambda \tau \nabla^2_w \eta(\bar{w})\partial_y \bar{w} )(f_4-\overline{\mathcal{M}}_4)\Bigg]\; dt \, dx \, dy+O(\varepsilon^3)\\\\
&=-\int_0^T \iint\frac{\lambda}{\varepsilon} \partial_x [\nabla_w \eta(\bar{w})\cdot((f_1-f_3)-(\overline{\mathcal{M}}_1-\overline{\mathcal{M}}_3))]\; dt \, dx \, dy\\
&-\int_0^T \iint\frac{\lambda}{\varepsilon}\partial_y  [\nabla_w \eta(\bar{w})\cdot ((f_2-f_4)-(\overline{\mathcal{M}}_2-\overline{\mathcal{M}}_4))]\; dt \, dx \, dy\\
& + \lambda^2 \tau \int_0^T \iint\partial_x [\nabla^2_w \eta(\bar{w}) \cdot \partial_x \bar{w} \cdot (f_1+f_3-(\overline{\mathcal{M}}_1+\overline{\mathcal{M}}_3))]\; dt \, dx \, dy\\
& + \lambda^2 \tau \int_0^T \iint\partial_y[\nabla^2_w \eta(\bar{w})\cdot \partial_y \bar{w}\cdot (f_2+f_4-(\overline{\mathcal{M}}_2+\overline{\mathcal{M}}_4))]\; dt \, dx \, dy+O(\varepsilon^3)\\\\
&=-\int_0^T \iint \partial_x[\nabla_w \eta(\bar{w})\cdot (m-\frac{A_1(\bar{w})}{\varepsilon}+2a \tau \lambda^2 \partial_x \bar{w})]\; dt \, dx \, dy\\
& - \int_0^T \iint\partial_y [\nabla_w \eta(\bar{w})\cdot (\xi-\frac{A_2(\bar{w})}{\varepsilon}+2a \tau \lambda^2 \partial_y \bar{w})]\; dt \, dx \, dy \\
& + \lambda^2 \tau \int_0^T \iint\partial_x [\nabla^2_w \eta(\bar{w}) \cdot \partial_x \bar{w}\cdot  (k-2a\bar{w})]\; dt \, dx \, dy\\
& + \lambda^2 \tau \int_0^T \iint\partial_y [\nabla^2_w \eta(\bar{w})\cdot  \partial_y \bar{w} \cdot (h-2a\bar{w})]\; dt \, dx \, dy+O(\varepsilon^3)
\end{align*}
\begin{align*}
&\begin{aligned} = - \iint  \int_0^T &
 \nabla_w \eta(\bar{w})\cdot [\partial_x m + \partial_y \xi - \frac{A_1(\bar{w})}{\varepsilon}-\partial_y \frac{A_2(\bar{w})}{\varepsilon}\\
&+2a\tau \lambda^2 \partial_{xx}\bar{w}+2a \tau \lambda^2 \partial_{yy}\bar{w}] \; dt \, dx \, dy
\end{aligned}\\
& -  \int_0^T \iint\nabla^2_w \eta(\bar{w}) \cdot \partial_x \bar{w} \cdot (m-\frac{A_1(w)}{\varepsilon}) \; dt \, dx \, dy \\
& -  \int_0^T \iint \nabla^2_w \eta(\bar{w})\cdot \partial_x \bar{w} \cdot (\frac{A_1(w)}{\varepsilon}-\frac{A_1(\bar{w})}{\varepsilon}) \; dt \, dx \, dy \\
&- 2a \tau \lambda^2  \int_0^T \iint\nabla^2_w\eta(\bar{w}) \cdot \partial_x \bar{w} \cdot \partial_x \bar{w} \; dt \, dx \, dy \\
& - \int_0^T \iint\nabla^2_w\eta(\bar{w}) \cdot \partial_y \bar{w} \cdot (\xi-\frac{A_2(w)}{\varepsilon}) \; dt \, dx \, dy \\
& -  \int_0^T \iint \nabla^2_w \eta(\bar{w}) \cdot \partial_y \bar{w}\cdot (\frac{A_2(w)}{\varepsilon}-\frac{A_2(\bar{w})}{\varepsilon}) \; dt \, dx \, dy \\
&-2a\tau \lambda^2  \int_0^T \iint\nabla^2_w \eta(\bar{w}) \cdot \partial_y \bar{w}\cdot \partial_y \bar{w} \; dt \, dx \, dy \\
& +  \lambda^2 \tau \int_0^T \iint \nabla^2_w\eta(\bar{w}) \cdot \partial_x \bar{w} \cdot (\partial_x k-2a \partial_x \bar{w}) + {\nabla^2_w \eta(\bar{w}) \cdot \partial_{xx}\bar{w}\cdot (k-2a\bar{w})} \; dt \, dx \, dy \\
& + \lambda^2 \tau  \int_0^T \iint \nabla^2_w \eta(\bar{w}) \cdot \partial_y \bar{w} \cdot  (\partial_y h-2a \partial_y \bar{w}) + \nabla^2_w \eta(\bar{w}) \cdot  \partial_{yy}\bar{w}\cdot (h-2a\bar{w}) \; dt \, dx \, dy \\
& + \lambda^2 \tau  \int_0^T \iint\nabla^3_w\eta(\bar{w}) (\partial_x \bar{w})^2 (k-2a \bar{w}) + \nabla^3_w\eta(\bar{w}) (\partial_y \bar{w})^2 (h-2a \bar{w}) \; dt \, dx \, dy + O(\varepsilon^3)\\\\
&\begin{aligned} & = -\int_0^T \iint\nabla_w \eta(\bar{w})\cdot [\partial_x m + \partial_y \xi - \frac{A_1(\bar{w})}{\varepsilon}-\partial_y \frac{A_2(\bar{w})}{\varepsilon}] \; dt \, dx \, dy\\
&-2a\tau \lambda^2 \int_0^T \iint \partial_{xx}\bar{w}+\partial_{yy}\bar{w}\; dt \, dx \, dy\\
\end{aligned}\\
& +\int_0^T \iint{\nabla^2_w\eta(\bar{w}) \cdot \partial_x \bar{w} \cdot \tau(\varepsilon^2 \partial_t m + \lambda^2 \partial_x k)} \; dt \, dx \, dy\\
& - \int_0^T \iint {\nabla^2_w\eta(\bar{w})\cdot \partial_x \bar{w} \cdot (\frac{A_1(w)}{\varepsilon}-\frac{A_1(\bar{w})}{\varepsilon})} \; dt \, dx \, dy\\
& - \int_0^T \iint 4a \tau \lambda^2 \nabla_w^2\eta(\bar{w}) \cdot \partial_x \bar{w}\cdot \partial_x \bar{w} \; dt \, dx \, dy\\
&+ \int_0^T \iint{\nabla^2_w \eta(\bar{w})\cdot  \partial_y \bar{w} \cdot \tau(\varepsilon^2 \partial_t \xi + \lambda^2 \partial_y h)} \; dt \, dx \, dy\\
&-\int_0^T \iint{\nabla^2_w\eta(\bar{w})\cdot  \partial_y \bar{w} \cdot (\frac{A_2(w)}{\varepsilon}-\frac{A_2(\bar{w})}{\varepsilon})} \; dt \, dx \, dy\\
&-\int_0^T \iint{4a\tau \lambda^2 \nabla^2_w \eta(\bar{w}) \cdot \partial_y \bar{w}\cdot \partial_y \bar{w}}  \; dt \, dx \, dy
\end{align*}
\begin{align*}
& + \lambda^2 \tau  \int_0^T \iint\nabla^2_w \eta(\bar{w}) \cdot \partial_x \bar{w} \cdot \partial_x k + \nabla^2_w \eta(\bar{w}) \cdot \partial_y \bar{w} \cdot \partial_y h  \; dt \, dx \, dy\\
& + \lambda^2 \tau \int_0^T \iint \nabla^2_w \eta(\bar{w})\cdot \partial_{xx} \bar{w} \cdot (k-2aw) + 2a\nabla^2_w \eta(\bar{w})\cdot \partial_{xx} \bar{w} \cdot (w-\bar{w})  \; dt \, dx \, dy\\
&+ \lambda^2 \tau \int_0^T \iint \nabla^2_w \eta(\bar{w})\cdot \partial_{yy} \bar{w} \cdot (h-2aw) + 2a \nabla^2_w \eta(\bar{w}) \cdot \partial_{yy}\bar{w} \cdot (w- \bar{w})  \; dt \, dx \, dy\\
& + \lambda^2 \tau  \int_0^T \iint\nabla^3_w\eta(\bar{w}) [(\partial_x \bar{w})^2 (k-2a \bar{w}) + (\partial_y \bar{w})^2 (h-2a \bar{w})] \; dt \, dx \, dy + O(\varepsilon^3).\\
\end{align*}
As an intermediate step, let us look at the sum
\begin{align*}
I_3+I_4&=2\tau \int_0^T \iint\nabla^2_w\eta(\bar{w}) \cdot \partial_x \bar{w}\cdot (\varepsilon^2 \partial_t m + \lambda^2 \partial_x k)  \; dt \, dx \, dy\\
& + 2\tau  \int_0^T \iint \nabla^2_w\eta(\bar{w}) \cdot \partial_y \bar{w}\cdot (\varepsilon^2 \partial_t \xi + \lambda^2 \partial_y h)  \; dt \, dx \, dy\\
& - \int_0^T \iint\nabla^2_w\eta(\bar{w})\cdot \partial_x \bar{w}\cdot (\frac{A_1(w)}{\varepsilon}-\frac{A_1(\bar{w})}{\varepsilon}-\frac{A_1'(\bar{w})}{\varepsilon}(w-\bar{w})) \; dt \, dx \, dy\\
& -\int_0^T \iint\nabla^2_w\eta(\bar{w}) \cdot \partial_y \bar{w}\cdot (\frac{A_2(w)}{\varepsilon}-\frac{A_2(\bar{w})}{\varepsilon}-\frac{A_2'(\bar{w})}{\varepsilon}(w-\bar{w})) \; dt \, dx \, dy\\
&+ \lambda^2 \tau \int_0^T \iint\nabla^2_w\eta(\bar{w}) \cdot [\partial_{xx}\bar{w} \cdot (k-2aw) + \partial_{yy}\bar{w} \cdot (h-2aw)] \; dt \, dx \, dy\\
&-4a \tau \lambda^2 \int_0^T \iint\nabla^2_w\eta(\bar{w}) \cdot \partial_x \bar{w} \cdot \partial_x \bar{w}+\nabla^2_w\eta(\bar{w}) \cdot \partial_y \bar{w} \cdot \partial_y \bar{w} \; dt \, dx \, dy\\
& + \lambda^2 \tau \int_0^T \iint\nabla^3_w\eta(\bar{w}) (\partial_x \bar{w})^2 (k-2a \bar{w}) + \nabla^3_w\eta(\bar{w}) (\partial_y \bar{w})^2 (h-2a \bar{w}) \; dt \, dx \, dy \\
&+ O(\varepsilon^3).
\end{align*}
We analyse each line separately.
\begin{itemize}
\item The first one can be written as 
$$\frac{\tau}{2a \lambda^2} \int_0^T \iint \nabla^2_w\eta(\bar{w}) \cdot (\varepsilon^2 \partial_t \bar{m}+\lambda^2 \partial_x \bar{k}) \cdot (\varepsilon^2 \partial_t {m}+\lambda^2 \partial_x {k}) \; dt \, dx \, dy + O(\varepsilon^3).$$
\item Similarly for the second line.
\item The third/fourth lines are equivalent to $$|\nabla^2_w \eta(\bar{w})|_{L_t^\infty L_x^\infty}  \int_0^T \iint |w-\bar{w}|^2 \; dt \, dx \, dy+O(\varepsilon^3).$$
\item The fifth line can estimated by
\begin{align*}
&c_1(|\nabla^2_w \eta(\bar{w})|_{L_t^\infty L_x^\infty})   \int_0^T \iint\varepsilon^2 |\partial_{xx} \bar{w}|^2+\varepsilon^2|\partial_{yy} \bar{w}|^2 \; dt \, dx \, dy\\
& + c_2(|\nabla^2_w \eta(\bar{w})|_{L_t^\infty L_x^\infty})  \int_0^T \iint\frac{|k-2aw|^2}{\varepsilon^2}+\frac{|h-2aw|^2}{\varepsilon^2} \; dt \, dx \, dy,
\end{align*}
where the first term is $O(\varepsilon^4)$, while the second one is absorbed by the dissipation in $I_1$.
\item The sixth term can be written as
\begin{align*}
&-\frac{\tau}{a \lambda^2} \int_0^T \iint\nabla^2_w\eta(\bar{w}) \cdot (\varepsilon^2 \partial_t \bar{m}+\lambda^2 \partial_x \bar{k}) \cdot  (\varepsilon^2 \partial_t \bar{m}+\lambda^2 \partial_x \bar{k})  \; dt \, dx \, dy\\
& -\frac{\tau}{a \lambda^2} \int_0^T \iint\nabla^2_w\eta(\bar{w}) \cdot (\varepsilon^2 \partial_t \bar{\xi}+\lambda^2 \partial_y \bar{h}) \cdot (\varepsilon^2 \partial_t \bar{\xi}+\lambda^2 \partial_y \bar{h})  \; dt \, dx \, dy + O(\varepsilon^3).
\end{align*}
\item The last term presents the following form
\begin{align*}
& \lambda^2 \tau \int_0^T \iint\nabla^3_w\eta(\bar{w}) (\partial_x \bar{w})^2 (k-2a \bar{w}) + \nabla^3_w\eta(\bar{w}) (\partial_y \bar{w})^2 (h-2a \bar{w})  \; dt \, dx \, dy\\
&= \lambda^2 \tau \int_0^T \iint \nabla^3_w\eta(\bar{w}) (\partial_x \bar{w})^2 (k-2aw)- 2a \nabla^3_w\eta(\bar{w}) (\partial_x \bar{w})^2 (w-\bar{w})\; dt \, dx \, dy\\
& + \lambda^2 \tau \int_0^T \iint\nabla^3_w\eta(\bar{w}) (\partial_y \bar{w})^2 (h-2a w) - 2a \nabla^3_w\eta(\bar{w}) (\partial_y \bar{w})^2 (w-\bar{w})\; dt \, dx \, dy\\
& \le c(|\nabla^2_w\eta(\bar{w})|_{L_t^\infty L_x^\infty})\int_0^T \iint|w-\bar{w}|^2 + \frac{|k-2aw|^2}{\varepsilon^2} + \frac{|h-2aw|^2}{\varepsilon^2} \; dt \, dx \, dy + O(\varepsilon^4),
\end{align*}
where the right-hand side is controlled by using the dissipation coming from $I_1$.
\end{itemize}
\begin{remark}
Denoting by $\mu_i(\nabla^2_w \eta(w)), \; \mu_i(\nabla^2_w \eta(\bar{w}))$ the eigenvalues of $\nabla^2_w \eta(w),$ $\nabla^2_w \eta(\bar{w})$ respectively, by simple calculations one gets that
$$\int_0^T |\mu_i(\nabla^2_w \eta(w(t)))-\mu_i(\nabla_w \eta(\bar{w}(t)))|_\infty \, dt \le c \int_0^T |\rho(t)-\bar{\rho}|_\infty \, dt =O(\varepsilon^2),$$
where the last equality follows from Lemma \ref{lemma_energy_estimates_KRM_paper}. Thus, we can write
\begin{align*}
&\frac{1}{2a\lambda^2}\int_0^T \iint(\nabla^2_w \eta(w) \cdot  (\varepsilon^2 \partial_t m + \lambda^2 \partial_x k)) \cdot (\varepsilon^2 \partial_t m + \lambda^2 \partial_x k) \; dt \, dx \, dy\\
&+\frac{1}{2 a \lambda^2}\int_0^T \iint(\nabla^2_w \eta(w) \cdot (\varepsilon^2 \partial_t \xi + \lambda^2 \partial_y h)) \cdot  (\varepsilon^2 \partial_t \xi + \lambda^2 \partial_y h) \; dt \, dx \, dy\\
&=\frac{1}{2a\lambda^2}\int_0^T \iint(\nabla^2_w \eta(\bar{w}) \cdot  (\varepsilon^2 \partial_t m + \lambda^2 \partial_x k)) \cdot (\varepsilon^2 \partial_t m + \lambda^2 \partial_x k) \; dt \, dx \, dy\\
&+\frac{1}{2 a \lambda^2} \int_0^T \iint (\nabla^2_w \eta(\bar{w}) \cdot (\varepsilon^2 \partial_t \xi + \lambda^2 \partial_y h)) \cdot  (\varepsilon^2 \partial_t \xi + \lambda^2 \partial_y h) \; dt \, dx \, dy+O(\varepsilon^3).
\end{align*}
\end{remark}
Now we consider the total sum, given by
\begin{align*}
I_1&+I_2+I_3+I_4 \le |\nabla^2_w \eta(\bar{w})|_{L_t^\infty L_x^\infty} \int_0^T \iint |w-\bar{w}|^2 \; dt \, dx \, dy\\
& -\frac{\tau}{2a\lambda^2}\int_0^T \iint\nabla^2_w\eta(\bar{w}) \cdot (\varepsilon^2 \partial_t m + \lambda^2 \partial_x k) \cdot (\varepsilon^2 \partial_t m + \lambda^2 \partial_x k)  \; dt \, dx \, dy\\
& -\frac{\tau}{2a\lambda^2}\int_0^T \iint\nabla^2_w\eta(\bar{w}) \cdot(\varepsilon^2 \partial_t \xi + \lambda^2 \partial_y h) \cdot (\varepsilon^2 \partial_t \xi + \lambda^2 \partial_y h)  \; dt \, dx \, dy\\
& + \frac{\tau}{a \lambda^2} \int_0^T \iint\nabla^2_w\eta(\bar{w}) \cdot(\varepsilon^2 \partial_t m + \lambda^2 \partial_x k)\cdot (\varepsilon^2 \partial_t \bar{m} + \lambda^2 \partial_x \bar{k})  \; dt \, dx \, dy\\
& + \frac{\tau}{a \lambda^2} \int_0^T \iint\nabla^2_w\eta(\bar{w}) \cdot (\varepsilon^2 \partial_t \xi + \lambda^2 \partial_y h) \cdot (\varepsilon^2 \partial_t \bar{\xi} + \lambda^2 \partial_y \bar{h})  \; dt \, dx \, dy\\
& - \frac{c(|\nabla^2_w\eta(w)|_{L_t^\infty L_x^\infty}(1-\frac{1}{\delta}))}{2a\tau \varepsilon^2}\int_0^T \iint |k-2aw|^2+|h-2aw|^2 \; dt \, dx \, dy\\
&-\dfrac{c(|\nabla^2_w\eta(w)|_{L_t^\infty L_x^\infty})}{(1-4a)\tau \varepsilon^2}\int_0^T \iint|4aw-(k+h)|^2\; dt \, dx \, dy+O(\varepsilon^3).\\
\end{align*}

The Gronwall inequality, together with the definition of $w$ in (\ref{w_def}), yields the following estimate

\begin{equation}
\label{zero_order_estimate}
\sup_{t \in [0, T^*]} \frac{\|\rho(t)-\bar{\rho}\|_0}{\varepsilon}+\|\rho \textbf{u}(t)-\bar{\rho}\bar{\textbf{u}}(t)\|_0 \le c \sqrt{\varepsilon},
\end{equation} 

where the local time $T^*$ is defined in (\ref{T_star_def}). 
\end{proof}

Now we prove Theorem \ref{Theorem_entropy_inequality_NS}.

\begin{proof}
We start by using the interpolation properties of Sobolev spaces, see \cite{Taylor}, for $0< s' < s$ and $t \in [0, T^*]$, which gives

\begin{equation}
\label{s'_rate}
\begin{aligned}
\|\rho \textbf{u}(t)-\bar{\rho}\bar{\textbf{u}}(t)\|_{s'} & \le \|\rho \textbf{u}(t)-\bar{\rho}\bar{\textbf{u}}(t)\|_0^{1-s'/s} \|\rho \textbf{u}(t)-\bar{\rho}\bar{\textbf{u}}(t)\|_s^{s'/s}\\
& \le c \varepsilon^{\frac{s-s'}{2s}} (M_0+ cM_0 e^{t|\textbf{u}|_{L_t^\infty L_x^\infty} })^{s'/s},
\end{aligned}
\end{equation}

where the last inequality follows by 
\begin{itemize}
\item the $H^s$- bound of the solution to the incompressible Navier-Stokes equations on the two-dimensional torus, i.e. $$\|\bar{\rho}\bar{\textbf{u}}(t)\|_s \le \|\bar{\rho}\textbf{u}_0\|_s \le M_0;$$
\item  the Gronwall inequality applied to estimate (\ref{s_integral_estimate_KRM_paper}), 
$$\|\rho \textbf{u}(t)\|_s \le cM_0 e^{c(|\rho|_{L_t^\infty L_x^\infty}, |\textbf{u}|_{L_t^\infty L_x^\infty})t}.$$
\end{itemize}

Taking $s'$ big enough, the Sobolev embedding theorem yields
\begin{equation}
\begin{aligned}
\label{L_infty_difference_rhou}
|\rho \textbf{u}(t)-\bar{\rho}\bar{\textbf{u}}(t)|_\infty & \le c_S \|\rho \textbf{u}(t)-\bar{\rho}\bar{\textbf{u}}(t)\|_{s'}\\
& \le c \varepsilon^{\frac{s-s'}{2s}} (M_0+ cM_0 e^{c(|\rho|_{L_t^\infty L_x^\infty}, |\textbf{u}|_{L_t^\infty L_x^\infty})t})^{s'/s},
\end{aligned}
\end{equation}

and so
\begin{align*}
|\rho \textbf{u}(t)|_\infty \le M_0 + c \varepsilon^{\frac{s-s'}{2s}} (M_0+ cM_0 e^{c(|\rho|_{L_t^\infty L_x^\infty}, |\textbf{u}|_{L_t^\infty L_x^\infty})t})^{s'/s}.
\end{align*}

Similarly, 
\begin{equation}
\label{L_infty_difference_rho}
\begin{aligned}
|\rho(t)-\bar{\rho}|_{\infty} & \le c_S\|\rho(t)-\bar{\rho}\|_{s'} \\
& \le  c_S\|\rho(t)-\bar{\rho}\|_0^{1-s'/s}  \|\rho(t)-\bar{\rho}\|_s^{s'/s}\\
& \le c \varepsilon^{\frac{3(s-s')}{2s}} (c \varepsilon M_0 e^{c(|\rho|_{L_t^\infty L_x^\infty}, |\textbf{u}|_{L_t^\infty L_x^\infty})t})^{s'/s},
\end{aligned}
\end{equation}
i.e.

\begin{equation}
\label{L_infty_rho}
|\rho(t)|_\infty \le \bar{\rho}+c \varepsilon^{\frac{3(s-s')}{2s}} (c \varepsilon M_0 e^{c(|\rho|_{L_t^\infty L_x^\infty}, |\textbf{u}|_{L_t^\infty L_x^\infty})t})^{s'/s}.
\end{equation}

Now, since
$$\textbf{u}-\bar{\textbf{u}}=\frac{1}{\rho}(\rho \textbf{u}-\bar{\rho}\bar{\textbf{u}})+\frac{\bar{\textbf{u}}}{\rho}(\bar{\rho}-\rho),$$

then from (\ref{L_infty_difference_rho})-(\ref{L_infty_difference_rhou})-(\ref{L_infty_difference_rho}),

\begin{equation}
\label{L_infy_difference_u}
\begin{aligned}
|\textbf{u}(t)-\bar{\textbf{u}}(t)|_\infty &\le \frac{1}{\bar{\rho}} \Bigg(c \varepsilon^{\frac{s-s'}{2s}} (M_0+ cM_0 e^{c(|\rho|_{L_t^\infty L_x^\infty}, |\textbf{u}|_{L_t^\infty L_x^\infty})t})^{s'/s}\\
&+c \varepsilon^{\frac{3(s-s')}{2s}} (c \varepsilon M_0 e^{c(|\rho|_{L_t^\infty L_x^\infty}, |\textbf{u}|_{L_t^\infty L_x^\infty})t})^{s'/s}\Bigg),
\end{aligned}
\end{equation}

i.e., 

\begin{equation}
\label{L_infty_bound_u}
\begin{aligned}
|\textbf{u}(t)|_\infty & \le M_0 + \frac{1}{\bar{\rho}} \Bigg(c \varepsilon^{\frac{s-s'}{2s}} (M_0+ cM_0 e^{c(|\rho|_{L_t^\infty L_x^\infty}, |\textbf{u}|_{L_t^\infty L_x^\infty})t})^{s'/s}\\
&+c \varepsilon^{\frac{3(s-s')}{2s}} (c \varepsilon M_0 e^{c(|\rho|_{L_t^\infty L_x^\infty}, |\textbf{u}|_{L_t^\infty L_x^\infty})t})^{s'/s}\Bigg).
\end{aligned}
\end{equation}

Recalling the definition of $T^*$ in (\ref{T_star_def}) and taking $M=4M_0,$
estimate (\ref{L_infty_bound_u}) implies that there exists $\varepsilon_0$ fixed such that, for $\varepsilon \le \varepsilon_0$ and $t \le T^*$,
$$|\textbf{u}(t)|_\infty \le M_0 + c \varepsilon^{\frac{1}{2}-\delta} <2M_0=\frac{M}{2},$$
for $0<\delta=\frac{s'}{2s} < \frac{1}{2}$. Similarly, for $t \le T^*$,
\begin{equation}
\label{better_estimate}
\begin{aligned}
\frac{|\rho(t)-\bar{\rho}|_\infty}{\varepsilon}+|\rho \textbf{u}(t)|_\infty \le M_0+c\varepsilon^{\frac{1}{2}-\delta}  < 2M_0=\frac{M}{2}.
\end{aligned}
\end{equation}

Now let us assume $T^* < T^\varepsilon$. Then, by definition (\ref{T_star_def}),
$$\frac{|\rho(T^*)-\bar{\rho}|_\infty}{\varepsilon}+|\rho \textbf{u}(T^*)|_\infty=4M_0=M.$$
On the other hand, estimate (\ref{better_estimate}) implies that there exists a fixed $\varepsilon_0$, depending on $M_0$ and small enough such that, for $\varepsilon \le \varepsilon_0$,
$$\frac{|\rho(T^*)-\bar{\rho}|_\infty}{\varepsilon}+|\rho \textbf{u}(T^*)|_\infty \le M_0 + c \varepsilon^{\frac{1}{2}-\delta} <2M_0.$$

Now, by contradiction one gets that $T^*=T^\varepsilon$ for $\varepsilon \le \varepsilon_0$. As a consequence, for $\varepsilon \le \varepsilon_0$ the solutions $(\rho^\varepsilon, \textbf{u}^\varepsilon)$ to the approximating system evaluated in $T^\varepsilon$ are bounded. This way, the Continuation Principle, see \cite{Majda}, implies that they are globally bounded in time.
Moreover, since the uniform bounds in Lemma \ref{lemma_energy_estimates_KRM_paper}  are based on the $L_t^\infty L_x^\infty$ boundedness of $(\rho^\varepsilon, \textbf{u}^\varepsilon)$, it turns out that they hold globally in time for $\varepsilon \le \varepsilon_0$.
In the end, we proved:
\begin{itemize}
\item the global in time existence and uniform boundedness of $(\rho^\varepsilon, \textbf{u}^\varepsilon)$ in $H^s(\mathbb{T}^2)$ for a fixed $\varepsilon \le \varepsilon_0$ depending on $M_0$;
\item the strong convergence in $[0, T]$, for any $T>0$, of the solutions $(\rho^\varepsilon, \textbf{u}^\varepsilon)$ to the approximating system (\ref{BGK_NS}) to the solutions $(\bar{\rho}, \bar{\textbf{u}})$ to the incompressible Navier-Stokes equations in $H^{s'}(\mathbb{T}^2)$, for $0<s'<s$ and $s>3$;
\item the rate of this strong convergence.
\end{itemize}
Finally, the convergence to the gradient of the limit incompressible pressure $\nabla P^{NS}$ in (\ref{real_NS}) is discussed in details in \cite{Bianchini3}.
\end{proof}

\section*{Acknowledgements}
The author is grateful to Roberto Natalini, for his constant support, and to Laure Saint-Raymond, for useful discussions.
%The author is grateful to Roberto Natalini and Laure Saint-Raymond for useful discussions and support.

\end{document}